\documentclass[10pt,reqno]{amsart}
\usepackage{amssymb,mathrsfs,graphicx,extpfeil,amsmath,bbm}
\usepackage{ifthen,latexsym,float,colortbl}
\usepackage[margin=1in]{geometry}
\usepackage{caption}
\usepackage{dsfont}
\usepackage{arydshln}
\usepackage{sidecap}
\usepackage{rotating,dsfont,stackengine}
\usepackage{hyperref}
\usepackage{refcheck}
\usepackage{enumerate}

\usepackage{colortbl}
\definecolor{black}{rgb}{0.0, 0.0, 0.0}
\definecolor{red}{rgb}{1.0, 0.5, 0.5}
\provideboolean{shownotes} % define a boolean variable
\setboolean{shownotes}{true} % defined as ``true'' or ``false''
\newcommand{\margnote}[1]{
	\ifthenelse{\boolean{shownotes}}%
	{\marginpar{\raggedright\tiny\texttt{#1}}}%
	{}%
}
\newcommand{\hole}[1]{
	\ifthenelse{\boolean{shownotes}}%
	{\begin{center} \fbox{ \rule {.25cm}{0cm} \rule[-.1cm]{0cm}{.4cm}
				\parbox{.85\textwidth}{\begin{center} \texttt{#1}\end{center}} \rule
				{.25cm}{0cm}}\end{center}} {} }

\graphicspath{{pics/}}
\title[Marle model for polyatomic gases]
{Relativistic BGK model of Marle for polyatomic gases near equilibrium}

\author[Hwang]{Byung-Hoon Hwang}
\address[Byung-Hoon Hwang]{\newline Department of Mathematics Education\newline
	Sangmyung University, 20 Hongjimun 2-gil, Jongno-Gu, Seoul 03016, Republic of Korea}
\email{bhhwang@smu.ac.kr}

\numberwithin{equation}{section}

\newtheorem{theorem}{Theorem}[section]
\newtheorem{lemma}{Lemma}[section]

\newtheorem{proposition}{Proposition}[section]
\newtheorem{remark}{Remark}[section]

\newcommand{\R}{\mathbb R}
\newcommand{\N}{\mathbb N}
\newcommand{\F}{\mathcal F}
\newcommand{\I}{\mathcal I}

\newcommand{\T}{\mathbb T}

\newcommand{\bq}{\begin{equation}}
\newcommand{\eq}{\end{equation}}

\newcommand{\pa}{\partial}

\newcommand{\into}{\int_{\T^3}}

\newcommand{\intr}{\int_{\R^3}}
\newcommand{\inti}{\int_0^\infty}

\newcommand{\intro}{\iint_{\R^3 \times \T^3}}

\makeatletter
\def\moverlay{\mathpalette\mov@rlay}
\def\mov@rlay#1#2{\leavevmode\vtop{%
		\baselineskip\z@skip \lineskiplimit-\maxdimen
		\ialign{\hfil$\m@th#1##$\hfil\cr#2\crcr}}}
\newcommand{\charfusion}[3][\mathord]{
	#1{\ifx#1\mathop\vphantom{#2}\fi
		\mathpalette\mov@rlay{#2\cr#3}
	}
	\ifx#1\mathop\expandafter\displaylimits\fi}
\makeatother

\makeatletter
\newcommand*{\rom}[1]{\expandafter\@slowromancap\romannumeral #1@}
\makeatother
\makeatletter
\@namedef{subjclassname@2020}{\textup{2020} Mathematics Subject Classification}
\makeatother
\begin{document}
	%%%%%%%%%%%%%%%%
	\allowdisplaybreaks
	
	\date{\today}
	

	\begin{abstract} 
  In this paper, we consider the direct application of
  the relativistic extended thermodynamics theory of polyatomic gases developed in [Ann. Phys. 377 (2017) 414--445] to the relativistic BGK model proposed by Marle. We present the perturbed Marle model around the generalized J\"{u}ttner distribution and investigate the properties of the linear operator. Then we prove the global existence and large-time behavior  of classical solutions when the initial data is sufficiently  close to a global equilibrium.  
	\end{abstract}
	
	\maketitle \centerline{\date}

  \tableofcontents

	%%%%%%%%%%%%%%%%%%%%%%%%%%%%%%%%%%%%%%%%%%%%%%%%%%%%%%%%%%%%%%%%%%%%%%%%%%%%%%%%%5
	%
	%
	%                        Section: Introduction 
	%
	%
	%%%%%%%%%%%%%%%%%%%%%%%%%%%%%%%%%%%%%%%%%%%%%%%%%%%%%%%%%%%%%%%%%%%%%%%%%%%%%%%%%
	\setcounter{equation}{0}
	\section{Introduction}

\subsection{Relativistic BGK model for polyatomic gases}
The BGK model \cite{BGK} is the most widely known relaxation-time approximation of the celebrated Boltzmann equation. As the Boltzmann equation has been developed for various gas systems, the BGK model has also been proposed by many researchers. In the relativistic framework, two types of model equations are well known, proposed by Marle \cite{Mar3,Mar2} and by Anderson and Witting \cite{AW} respectively. The difference between them is the way they interpret macroscopic quantities. To be precise, the Marle model is based on the Eckart frame \cite{Eckart}, while the Anderson-Witting model adopts the Landau-Lifshitz frame \cite{LL}. Meanwhile, in 2017, Pennisi and Ruggeri \cite{PR2} developed a relativistic extended thermodynamics theory of rarefied polyatomic gases where the J\"{u}ttner distribution \cite{Juttner} (also called the relativistic Maxwellian) was generalized to polyatomic molecules for the first time. Using the generalized J\"{u}ttner distribution, they \cite{PR} suggested another relativistic BGK model of the Anderson-Witting type for both monatomic and polyatomic gases. These three types of BGK models have been applied to various fields related to the relativistic kinetic theory of gases in a complementary manner.

In this paper, we are interested in the direct application of
the relativistic extended thermodynamics \cite{PR2} to the relativistic BGK model of Marle \cite{Mar3,Mar2}. We let $F\equiv F(x^\mu , p^\mu,\mathcal{I} )$ be a distribution function describing the number density of a polyatomic gas on the phase space point $(x^\mu,p^\mu)$ with the internal energy $\mathcal{I}\ge 0$ due to the rotation and vibrations of particles.  The space-time coordinate $x^\mu$ and the four-momentum $p^\mu$ are given by
$$
x^\mu:=(ct,x)\in\mathbb{R}_+\times \T^3,\quad p^\mu:=(\sqrt{(mc)^2+|p|},p)\in\mathbb{R}_+\times \mathbb{R}^3,
$$
where $c$ is the speed of light and $m$ the rest mass of a particle, and  $\tau$ stands for the relaxation time in the rest frame. The Minkowski metric tensor $g_{\mu\nu}$ and its inverse $g^{\mu\nu}$  are given by 
$$
g_{\mu\nu}=g^{\mu\nu}=\text{diag}(1,-1,-1,-1),
$$
where Greek indices $\mu$ and $\nu$ run from $0$ to $3$.
We use the raising and lowering indices:
$$
g_{\mu\nu}p^\nu=p_\mu,\qquad g^{\mu\nu}p_\nu=p^\mu,
$$
which implies $p_\mu=(p^0,-p)$ by the Einstein convention. Accordingly, the Minkowski inner product is given by
$$
p^\mu q_\mu=p_\mu q^\mu=p^0q^0-\sum_{i=1}^3 p^iq^i.
$$
Next, we notice that the macroscopic descriptions of F rely on the particle-particle $V^\mu$ and the energy-momentum tensor $T^{\mu\nu}$ \cite{PR2}:
\begin{equation}\label{VT}
V^\mu=mc\int_0^\infty\int_{\mathbb{R}^3} p^\mu F \phi(\mathcal{I}) \,d\mathcal{I}\,\frac{dp}{p^0},\qquad T^{\mu\nu}=\frac{1}{mc} \int_0^\infty\int_{\mathbb{R}^3} p^\mu p^\nu F\left( mc^2 + \mathcal{I} \right)   \phi(\mathcal{I})\,d\mathcal{I}\,\frac{dp}{p^0},
\end{equation}
where $\phi(\mathcal{I})\ge 0$ is the state density so that  $\phi(\mathcal{I}) \, d  \mathcal{I}$ represents the number of the internal states of a molecule having the internal energy between $\mathcal{I}$ and $\mathcal{I}+d \mathcal{I}$. The form of $\phi(\mathcal{I})$ may be given in various ways depending on the physical context, but it must provide an exact non-relativistic limit. We refer to \cite[Sec 4.1]{PR2} for a discussion on the choice of $\phi$ and \cite{CPR1} for the monatomic limit. In this paper, we have set $\phi$ to the simple form
\begin{equation}\label{phi}
\phi(\mathcal{I})=\mathcal{I}^{\frac{D-2}{2}},
\end{equation}
where $D>0$ represents the contribution to the degrees of freedom related to the rotation and vibration of a polyatomic molecule. The integrability properties related to \eqref{VT} with \eqref{phi} can be found in \cite{CPR0}. Going back to \eqref{VT}, we decompose $V^\mu$ of \eqref{VT} based on the Eckart frame \cite{Eckart,PR2} as
 $$
 V^\mu =mnu^\mu
 $$
 where $n$ denotes the number density and $u^\mu=(\sqrt{c^2+|u|^2},u)\in \R_+\times \R^3$ the four-velocty, defined as
 \begin{align}\label{n,u}
 \begin{split}
 n^2&:=\frac{1}{(mc)^2}V^\mu V_\mu=\left(\inti\intr  F \phi(\mathcal{I}) \,d\mathcal{I}\,dp\right)^2-\sum_{i=1}^3\left(\inti\intr p^i F \phi(\mathcal{I}) \,d\mathcal{I}\,\frac{dp}{p^0}\right)^2,\cr 
 u^\mu &:=\frac{1}{mn}V^\mu=\frac {c}{n}\inti\intr p^\mu F \phi(\mathcal{I}) \,d\mathcal{I}\,\frac{dp}{p^0},
 \end{split}\end{align}
 respectively. 
% It is easy to check that 
% $$
% n=\frac 1c\inti\intr (p^\mu u_\mu)F \phi(\mathcal{I}) \,d\mathcal{I}\,\frac{dp}{p^0}.
% $$

Now we are ready to introduce the relativistic BGK model of Marle for polyatomic gases:
\begin{align}\label{PMarle}
\pa_t F+\frac{cp}{p^0}\cdot\nabla_x F&=\frac{cm}{\tau (1+\frac{\I}{mc^2})p^0}(\mathcal{F}_E-F).
\end{align}
Here $\mathcal{F}_E$ is the generalized J\"{u}ttner distribution derived in \cite{PR2}:
\begin{equation*} 
 \F_E(n,u,\gamma;p,\I)=\frac{n}{M(\gamma)}e^{-\left(1+\frac{\mathcal{I}}{mc^2}\right)\frac{\gamma}{mc^2} u^\mu p_\mu }
\end{equation*}
which describes the state of a gas in equilibrium depending on the parameters $n$, $u$, and $\gamma>0$.  It was shown in \cite[Sec.4]{PR2} that $\mathcal{F}_E$ is the distribution  maximizing  the relativistic entropy:  
$$
h=-ck_B  \inti\intr  u_\mu p^\mu F\ln F\phi(\I)\,d\I \frac{dp}{p^0}
$$
where $k_B$ is the Boltzmann constant. The  parameters $n$ and $u^\mu$ are given by \eqref{n,u}, and $\gamma:=\frac{mc^2}{k_BT}$ is the ratio between the rest energy of a particle and the product of the Boltzmann constant and the equilibrium temperature $T$. In order for \eqref{PMarle} to satisfy the fundamental properties of the kinetic theory of gases, $\gamma$ must be determined through the nonlinear relation:
\begin{align}\label{gamma}
\frac{\widetilde{M}(\gamma)}{M(\gamma)}&=\frac {1}{n}\inti\intr  F\Big(1+\frac{\I}{mc^2}\Big)^{-1}\phi(\mathcal{I}) \,d\mathcal{I} \frac{dp}{p^0},
\end{align}
where $M(\gamma)$ and $\widetilde{M}$ denote
\begin{align*}  
M(\gamma) &=\int_0^\infty\int_{\mathbb{R}^3}  e^{-\left(1+\frac{\mathcal{I}}{mc^2}\right)\frac{\gamma}{mc} p^0 }  \phi(\mathcal{I}) \,d\mathcal{I}\,dp,\cr 
\widetilde{M}(\gamma)&=\inti\intr  e^{-\left(1+\frac{\mathcal{I}}{mc^2}\right)\frac{\gamma}{mc}p^0 }\Big(1+\frac{\I}{mc^2}\Big)^{-1}\phi(\mathcal{I}) \,d\mathcal{I} \frac{dp}{p^0}. 
\end{align*}
Note from \cite{Hwang24} that there is a one-to-one correspondence between both sides of \eqref{gamma}, which guarantees that $\gamma$ is uniquely determined  by $F$. Due to the definition of $n,u^\mu$ and $\gamma$,  the generalized J\"{u}ttner distribution  $\F_E$ satisfies
\begin{align}\label{cancellation}
\begin{split}
\inti\intr \frac{1}{(mc^2+\I)}(\mathcal{F}_E-F)\phi(\I)\,d\I \frac{dp}{p^0}=0,\quad \inti\intr  p^\mu(\mathcal{F}_E-F)\phi(\I)\,d\I \frac{dp}{p^0}=0
\end{split}\end{align}
which enables us to obtain the conservation laws of $V^\mu$ and $T^{\mu\nu}$ for \eqref{PMarle}:
\begin{align*}
\pa_\mu V^\mu &:=m\inti\intr \{\pa_t F+ \frac{cp}{p^0}\cdot\nabla_x F \}  \phi(\mathcal{I}) \,d\mathcal{I}\,dp=0,\cr 
\pa_\nu T^{\mu\nu}&:=\frac{1}{mc^2} \inti\intr p^\mu \{\pa_t F+ \frac{cp}{p^0}\cdot\nabla_x F \}\left( mc^2 + \mathcal{I} \right)\phi(\mathcal{I})\,d\mathcal{I}\,dp=0,
\end{align*}
and the $H$-theorem:  
$$
\pa_\mu h^\mu:=  -k_B  \inti \intr\{ \pa_t+\frac{cp}{p^0}\cdot\nabla_x\} F\ln F\phi(\I)\,d\I dp \ge 0.
$$

\subsection{Main result}
The aim of this paper is to establish the global-in-time existence  and large-time behavior of unique smooth solution to \eqref{PMarle} when the initial data is sufficiently close to a global equilibrium state. To do this, we let $\F_0$ be a global equilibrium distribution defined by
$$
\mathcal{F}_0\equiv \mathcal{F}(1,0,\gamma_0; p,\I) =\frac{1}{M(\gamma_0)}e^{-\gamma_0 (1+\I)p^0},\quad \mbox{where } M(\gamma_0)=\int_0^\infty\int_{\mathbb{R}^3} e^{-\gamma_0(1+\I) p^0}\phi(\I)\,d\I dp,
$$
for some positive constant $\gamma_0$, where all
the physical constants are set to unity for notational simplicity. Decompose the distribution $F$ into 
$$
F(t,x,p,\I)=\mathcal{F}_0+\sqrt{\mathcal{F}_0}f(t,x,p,\I),\qquad F_0(x,p,\I):=F(t,x,p,\I)\big|_{t=0} =\mathcal{F}_0+\sqrt{\mathcal{F}_0}f_0(x,p,\I)
$$
where $f$ stands for the perturbation of $\F_0$ with initial data $f_0$. Substituting it into the Cauchy problem for the  Marle model \eqref{PMarle}, the perturbation $f$ verifies
\begin{align}\label{perturbed} \begin{split}
\pa_tf  +\frac{p}{p^0}\cdot\nabla_x f &=\frac{1}{(1+\I)p^0}\{\mathcal{P}_0(f)-f\}+\Gamma(f)\cr 
f_0(x,p,\I)&=f(0,x,p,\I),
\end{split}\end{align}
see Proposition \ref{linearization} for details. To introduce our main result, we define some notations and notational conventions. 
		\begin{itemize}
			\item Throughout this paper, we use the notation $C$ to denote generic positive constants. When it is necessary
			to reveal a dependency on specific quantities, we denote it as $C_{a,b,c\cdots}$.
			\item For functions $f(x,p,\I)$ and $g(p,\I)$, we denote by $\|f\|_{L^2}$, $\|g\|_{L^2}$ the usual $L^2\big(\T^3\times \R^3\times[0,\infty)\big)$ and $L^2\big(\R^3\times [0,\infty)\big)$ norms respectively. Also, we define weighted inner products
			\begin{align*}
			\langle f,g\rangle_{L^2_{p,\I}}&=\inti\intr f(p,\I) g(p,\I) \phi(\I) \,d\I dp,\cr  	
			\langle f,g\rangle_{L^2_{x,p,\I}}&=\inti\intro f(x,p,\I) g(x,p,\I) \phi(\I) \,d\I dpdx			
			\end{align*}
				and  weighted $L^2$ norms 
			\begin{align*}
			\|f\|^2_{L^2_{p,\I}}&=\inti\intr |f(p,\I)|^2\phi(\I)\,d\I  dp,\cr \|f\|^2_{L^2_{x,p,\I}}&=\inti\intro |f(x,p,\I)|^2\phi(\I)\,d\I dp dx.
			\end{align*}
			\item For $z=(z_1,\cdots,z_n)\in \R^n$, We use $\mathbb{P}(z)$ to denote a generic polynomial
			$$
			\mathbb{P}(z)=\sum_{i}C_{i}z^{\kappa_{1_i}}_{1}\cdots z^{\kappa_{n_i}}_{n}
			$$
			where $\kappa_{j_i}$ is a non-negative integer. 
			\item We use a multi-index $\alpha=[\alpha^{0},\alpha^{1},\alpha^{2},\alpha^{3}]$ to denote 
			$$
			\partial^{\alpha}=\partial_{t}^{\alpha^{0}}\partial_{x^{1}}^{\alpha^{1}}\partial_{x^{2}}^{\alpha^{2}}\partial_{x^{3}}^{\alpha^{3}},
			$$
			where the length of $\alpha$ is defined by $|\alpha|:=\alpha^0+\alpha^1+\alpha^2+\alpha^3$.
			\item We define the energy functional $\mathcal{E}(f)(t)$ by
			\begin{align*}
			\mathcal{E}(f)(t)=\sum_{|\alpha| \le N}\|\partial^{\alpha} f\|^{2}_{L^2_{x,p,\I}}.
			\end{align*}
		\end{itemize}
Then the main result of this paper is stated as follows.
	\begin{theorem}\label{main}
		Let $N\geq 3$. Assume that the initial data $F_0$ is non-negative and it shares the same total mass, momentum, and energy as the global equilibrium $\mathcal{F}_0$, i.e. 
		$$
		\inti\intro f_{0}\sqrt{\mathcal{F}_0} \phi(\I)\,d\I dpdx=0,\quad \inti\intro  (1+\I)p^\mu f_{0}\sqrt{\mathcal{F}_0} \phi(\I)\,d\I dpdx=0.
		$$  If $ \mathcal{E}(f_0)(t)$ is sufficiently small, then there exists a global-in-time unique solution $F(t,x,p,\I)\ge 0$ to the relativistic BGK model \eqref{PMarle} satisfying
		\begin{align*}
\mathcal{E}(f)(t)\le e^{-\lambda_0 t}\mathcal{E}(f_0)
		\end{align*}
for some positive constant $\lambda_0$.	
		%\noindent (4) $F$ converges weakly to $\F_0$ in $H=H^N(\mathbb{R}^3\times\mathbb{R}^3)$ in the following sense:\newline
		%For any $\{t_n\}$ such that $t_n\rightarrow\infty$ as $n\rightarrow\infty$, there exists a subsequence of $\{t_n\}$, which we still denote as $\{t_n\}$ for simplicity,
		%such that
		%\[
		%\lim_{n\rightarrow \infty}\int_{\mathbb{R}^3\times \mathbb{R}^3}\frac{F(x,q,t_n)-\F_0}{\sqrt{\F_0}}\phi(x,q)dxdq=0,
		%\]
		%for any $\phi\in H^*$.
	\end{theorem}
	\begin{remark}  Note from \cite{Hwang24} that the unique existence of the parameter $\gamma$ appearing in \eqref{gamma} is also guaranteed if the state density $\phi(\I)$ satisfies (1) $\phi(\I)$ does
not grow exponentially, (2) $\I\phi(\I)$=0 when $\I=0$, and (3) $\phi(\I)+\I\phi^\prime(\I)$ is governed by $(1+\I)\phi (\I)$ up to constant multiplication. This class of integrable functions $\phi(\I)$ seems also applicable to Theorem \ref{main} since it does not cause any integrability issues in obtaining technical lemmas due to the presence of the global equilibrium distribution $\F_0$.
	\end{remark}
\begin{remark}
The main  ingredient in the proof of Theorem \ref{main}  is to handle the degeneracy in the coercivity estimate of the linear operator 
by analyzing the so-called micro-macro decomposition developed in \cite{Guo whole,Guo VMB}. In the perturbed relativistic BGK model of Marle \eqref{perturbed}, the linear part is given by $\{(1+\I)p^0\}^{-1}\{\mathcal{P}_0(f)-f\}$ where $\mathcal{P}_0$ denotes 
\begin{align*}
&\mathcal{P}_0(f)=\left\{\int_0^\infty\int_{\mathbb{R}^3}  f\sqrt{\mathcal{F}_0} \phi(\mathcal{I}) \,d\mathcal{I}\,dp\right\}\sqrt{\F_0} +\gamma_0  \left\{\int_0^\infty\int_{\mathbb{R}^3} p  f\sqrt{\mathcal{F}_0}  \phi(\mathcal{I}) \,d\mathcal{I}\,\frac{dp}{p^0}\right\}\cdot (1+\I)p\sqrt{\F_0} \cr 
&+\frac{M^2(\gamma_0)}{M^2(\gamma_0)+M'(\gamma_0)\widetilde{M}(\gamma_0)}\biggl\{  \inti\intr  \biggl(\frac{1}{(1+\I)p^0}-\eta_0\biggl)f\sqrt{\mathcal{F}_0}  \phi(\I) d\I dp\biggl\}\biggl(\frac{M^\prime(\gamma_0 )}{M(\gamma_0 )}+(1+\I)  p^0\biggl)  \sqrt{\F_0} .
\end{align*}
Notice that the coefficients of $\mathcal{P}_0$ are moments with respect to $\{ 1/(1+\I)p^0, p/p^0 ,1\}$,  while the conserved quantities  consist of moments with respect to $\{1,(1+\I)p^\mu\}$, see Theorem \ref{thm8} (3). This incoherence makes it impossible to apply the Poincar\'{e} inequality to the coefficients of $P_0$ and hence the lowest order term in the micro-macroscopic decomposition cannot be controlled. To avoid this difficulty, we introduce an orthonormal projection $\mathcal{P}$ onto $span\{\sqrt{\F_0},(1+\I)p^\mu\sqrt{\F_0}\}$ that is consistent with the conservation laws, and show the coercive property of the linear operator without using $\mathcal{P}_0$:
$$
\langle L(f),f\rangle_{L^2_{p,\I}} \le -\lambda \|(I-\mathcal{P})f\|^2_{L^2_{p,\I}},
$$  
which is motivated by \cite[Lemma 3]{Guo03} where the classical solutions for the  Boltzmann equation was addressed under the Grad angular cutoff assumption.

\end{remark}
	%%%%%%%%%%%%%%%%%%%%%%%%%%%%%%%%%%%%%%%%%%%%%%%%%%%%%%%%%%%%%%%%%%%%%%%%%%%%%%%%%%%%%%%%%%%%%%%%%%%%%%%%%%%%%%%%%%%%%%%%
\subsection{Brief history}
The first mathematical study of the relativistic BGK model can be traced back to Bellouquid et al. \cite{BCNS} where the unique determination of equilibrium parameters, asymptotic limits, and the linearized problem were addressed for the monatomic Marle model. Then, Bellouquid et al. \cite{BNU} proved the global existence and large-time behavior of solutions when the initial data is sufficiently close to the global equilibrium distribution. The global existence of weak solutions was established by Calvo et al. \cite{CJS} where the velocity averaging lemma was generalized in the relativistic context. Hwang and Yun \cite{HY1} studied the stationary solutions in a slab by applying the Banach fixed point theorem. Recently, Sun \cite{Sun} proved the finite propagation speed of solutions for the Marle model near equilibrium. In the case of the relativistic BGK model of Anderson and Witting, we refer to \cite{Hwang,HY2} for the unique determination of equilibrium parameters and the near-equilibrium solutions, \cite{HY3} for the stationary problems, and \cite{HLY} for the explicit solutions to the Anderson-Witting model for massless particles in the FLRW spacetime. On the other hand, since the pioneering work of Pennisi and Ruggeri \cite{PR2}, the relativistic extended thermodynamics theory of rarefied polyatomic gases has been successfully developed, see \cite{CP,CPR0,CPR1,CPR2,CPR3,PR5,PR6}. However, except for the work by Hwang et al. \cite{HRY22} that proved the global existence and large-time behavior of near-equilibrium solutions for the model of Pennisi and Ruggeri, existence theory for the polyatomic BGK-type models has not been reported so far.

Compared to the relativistic BGK models, much more mathematical topics have been addressed for the relativistic Boltzmann equations:  For the spatially homogeneous case,  the propagation of upper bounds of solutions was studied in \cite{JSY,JY21} and the relativistic version of the Povzner inequality was derived in \cite{Strain Yun} respectively. In \cite{A,JY}, the regularizing effects of the collision operator were addressed. The existence of solutions with large data can be found in \cite{Dud3,Jiang2,Jiang1}. We also refer to  \cite{Cal,Strain2} for the Newtonian limits and \cite{Guo-Xiao21,SS} for the hydrodynamic limit.

 \noindent\newline

The rest of this paper is organized as follows.  In Section 2, we present the details of the linearization procedure of the Marle model for a polyatomic molecule. In Section 3, we provide the technical lemmas for estimating the nonlinear operator.  Section 4 is devoted to the proof of Theorem \ref{main}.
%%%%%%%%%%%%%%%%%%%%%%%%%%%%%%%%%%%%%%%%%
%
%
%
%
%
%
%%%%%%%%%%%%%%%%%%%%%%%%%%%%%%%%%%%%%%%%%	
\section{Linearization around the global equilibrium}
Let
$$
F=\mathcal{F}_0+\sqrt{\mathcal{F}_0}f.
$$
where $\mathcal{F}_0$ represents the global equilibrium state, given by
$$
\mathcal{F}_0\equiv \mathcal{F}(1,0,\gamma_0; p,\I) =\frac{1}{M(\gamma_0)}e^{-\gamma_0 (1+\I)p^0},\qquad \mbox{where } M(\gamma_0)=\int_0^\infty\int_{\mathbb{R}^3} e^{-\gamma_0(1+\I) p^0}\phi(\I)\,d\I dp.
$$
and $f$ is the perturbation part.
The aim of this section is to analyze the linearized terms of \eqref{PMarle} with respect to $f$. For this, we first recall that
the J\"{u}ttner distribution $\mathcal{F}_E$ is a function of $(n,u,\gamma)$ where the variable $\gamma$ is in one-to-one correspondence with $F$ by the nonlinear relation \eqref{gamma}:
$$
\Big( \frac{\widetilde{M}}{M}\Big)(\gamma)=\frac{1}{n}\inti\intr  F(1+\I)^{-1} \phi(\I)\, d\I\frac{dp}{p^0}.
$$
Thus, letting 
$$
\eta(t,x):=\frac{1}{n}\inti\intr  F(1+\I)^{-1} \phi(\I)\, d\I\frac{dp}{p^0}\quad \mbox{and} \quad \mathcal{X}(z):=\left( \widetilde{M}/M\right)^{-1}(z),
$$
 one can rewrite \eqref{gamma} as
$$
\gamma(t,x)=\mathcal{X}\left(\eta(t,x)\right).
$$
This implies that $\mathcal{F}$ can be regarded as a function of $(n,u,\eta)$ without any risk. The new variable $\eta$ can be also founded in  \cite[Sec.4]{BCNS} where the relativistic BGK model of Marle was addressed for a monatomic molecule. Using this, we introduce the transitional distribution 
\begin{equation}\label{transitional}
\mathcal{F}_\theta:=\mathcal{F}(n_\theta, u_\theta,\eta_\theta;p,\I) \quad \mbox{with }(n_\theta, u_\theta,\eta_\theta;p)=\theta(n, u,\eta)+(1-\theta)(1,0,\eta_0)
\end{equation}
where $\eta_0$ is a constant defined by
\begin{equation}\label{eta0}
\eta_0=\inti\intr  \F_0(1+\I)^{-1} \phi(\I)\, d\I\frac{dp}{p^0}\equiv \frac{\widetilde{M}(\gamma_0)}{M(\gamma_0)}.
\end{equation}
\begin{proposition}\label{linearization}
	Let $
	F=\mathcal{F}_0+\sqrt{\mathcal{F}_0}f$. Then the relativistic BGK model \eqref{PMarle} can be written as
	\begin{align}\label{target_linearized}\begin{split}
	 \pa_tf  +\frac{p}{p^0}\cdot\nabla_x f =\frac{1}{(1+\I)p^0}\left\{\mathcal{P}_0(f)-f \right\}+\Gamma(f).
	\end{split}\end{align}
	Here $\mathcal{P}_0$ is the operator defined as
\begin{align*}
\mathcal{P}_0(f)&=\left\{\int_0^\infty\int_{\mathbb{R}^3}  f\sqrt{\mathcal{F}_0} \phi(\mathcal{I}) \,d\mathcal{I}\,dp\right\}\sqrt{\F_0}+\gamma_0  \left\{\int_0^\infty\int_{\mathbb{R}^3} p  f\sqrt{\mathcal{F}_0}  \phi(\mathcal{I}) \,d\mathcal{I}\,\frac{dp}{p^0}\right\}\cdot p\left(1+\I\right) \sqrt{\F_0}\cr 
&+\frac{M^2(\gamma_0)}{M^2(\gamma_0)+M'(\gamma_0)\widetilde{M}(\gamma_0)}\left\{  \inti\intr  \left(\frac{1}{(1+\I)p^0}-\eta_0\right)f\sqrt{\mathcal{F}_0}  \phi(\I)\, d\I dp\right\}\left(\frac{M^\prime(\gamma_0 )}{M(\gamma_0 )}+(1+\I)  p^0\right)\sqrt{\F_0}
\end{align*}
	and 	$\Gamma$ is the nonlinear operator with respect to $f$ defined as
\begin{align}\label{Gamma}
\begin{split}
\Gamma(f)&=\frac{N_n \sqrt{\F_0}}{(1+\I)p^0}+\gamma_0  N_u\cdot \frac{p}{p^0} \sqrt{\F_0}+\frac{M^2(\gamma_0)}{M^2(\gamma_0)+M'(\gamma_0)\widetilde{M}(\gamma_0)}\left(\frac{1}{(1+\I)p^0}\frac{M^\prime(\gamma_0 )}{M(\gamma_0 )}+1\right)  N_\eta\sqrt{\F_0}\cr 
&+\frac{1}{(1+\I)p^0\sqrt{\F_0}}\int_0^1 (1-\theta) (n-1,u,\eta-\eta_0)\nabla^2_{(n_\theta, u_\theta,\eta_\theta)}\mathcal{F}_\theta(n-1,u,\eta-\eta_0)^T \,d\theta.
\end{split}
\end{align}
	where $N_n$, $N_u$ and $N_\eta$ are given by \eqref{Nn}, \eqref{Nu} and \eqref{Neta} respectively.
\end{proposition}

\begin{proof}
Since $\F_E$ is independent of $t$ and $x$, it suffices to show that 
$$
\frac{1}{\sqrt{\F_0}}(\mathcal{F}_E-F)=\mathcal{P}_0(f)-f+\Gamma(f)
$$
We first use the notation \eqref{transitional} to see that
	\begin{align*}
	\mathcal{F}_E-F&=  \mathcal{F}_E-\mathcal{F}_0-\sqrt{\mathcal{F}_0}f\cr 
	&\equiv   \mathcal{F}_1-\mathcal{F}_0-\sqrt{\mathcal{F}_0}f .
	\end{align*}
	Applying the Taylor expansion to $\mathcal{F}_1-\mathcal{F}_0$, one finds
	\begin{align*}
	\mathcal{F}_1-\mathcal{F}_0&=\frac{d}{d\theta} \left\{\mathcal{F}(n_\theta, u_\theta,\eta_\theta)\right\}\big|_{\theta=0}+\int_0^1 (1-\theta) \frac{d^2}{d\theta^2}\left\{\mathcal{F}(n_\theta, u_\theta,\eta_\theta)\right\} \,d\theta
	\end{align*}
	where 
	\begin{align}\label{theta}
	\frac{d}{d\theta} \left\{\mathcal{F}(n_\theta, u_\theta,\eta_\theta;p)\right\}\big|_{\theta=0}&=\frac{d \mathcal{F}_\theta}{dn_\theta}\biggl|_{\theta=0}(n-1)+\nabla_{u_\theta}\mathcal{F}_\theta\big|_{\theta=0}\cdot u+\frac{d \mathcal{F}_\theta}{d\eta_\theta}\biggl|_{\theta=0}(\eta-\eta_0)
	\end{align}
	and 
	$$
	\frac{d^2}{d\theta^2}\left\{\mathcal{F}(n_\theta, u_\theta,\eta_\theta;p)\right\}=(n-1,u,\eta-\eta_0)\nabla^2_{(n_\theta, u_\theta,\eta_\theta)}\mathcal{F}_\theta(n-1,u,\eta-\eta_0)^T.
	$$
It is straightforward that
	$$
	\frac{d \mathcal{F} }{dn }=\frac{1}{n }\mathcal{F} ,\quad \nabla_{u }\mathcal{F} =-\left(1+\I\right)\gamma\left(\frac{p^0}{u^0}u-p \right)\mathcal{F} , 
	$$
	and
	$$
	 \frac{d \mathcal{F} }{d\eta }=\frac{d\gamma }{d\eta }\frac{d \mathcal{F} }{d\gamma}	=\frac{M^2(\gamma)}{M^2(\gamma)+M'(\gamma)\widetilde{M}(\gamma)}\left(\frac{M^\prime(\gamma )}{M(\gamma )}+(1+\I)  u ^\mu p_\mu\right)  \mathcal{F} .
	$$
	where we used
	$$
	\frac{d\gamma }{d\eta }=\frac{1}{\big(  \widetilde{M}/M\big)^{\prime}(\gamma )}=-\frac{M^2(\gamma)}{M^2(\gamma)+M'(\gamma)\widetilde{M}(\gamma)}.
	$$
Inserting this into \eqref{theta}, we get
\begin{align}\label{theta2}
\begin{split}
	&\frac{d}{d\theta} \left\{\mathcal{F}(n_\theta, u_\theta,\eta_\theta;p)\right\}\big|_{\theta=0}\cr &=(n-1)\F_0+\left(1+\I\right)\gamma_0  u\cdot p \mathcal{F}_0+\frac{M^2(\gamma_0)}{M^2(\gamma_0)+M'(\gamma_0)\widetilde{M}(\gamma_0)}\left(\frac{M^\prime(\gamma_0 )}{M(\gamma_0 )}+(1+\I)  p^0\right)  (\eta-\eta_0)\mathcal{F}_0.
\end{split}\end{align}
Next, we decompose $n-1$, $u$ and $\eta-\eta_0$ into the linear parts and nonlinear parts with respect to $f$.
\newline
$\bullet$ (Linearization of $n-1$):	By definition, we have  
	\begin{align*}
	n&=\left\{\left(\int_0^\infty\int_{\mathbb{R}^3}  (\mathcal{F}_0+\sqrt{\mathcal{F}_0}f) \phi(\mathcal{I}) \,d\mathcal{I}\,dp\right)^2-\sum_{i=1}^3\left(\int_0^\infty\int_{\mathbb{R}^3} p^i (\mathcal{F}_0+\sqrt{\mathcal{F}_0}f) \phi(\mathcal{I}) \,d\mathcal{I}\,\frac{dp}{p^0}\right)^2 \right\}^{1/2}\cr
	&=\left\{\left(1+\int_0^\infty\int_{\mathbb{R}^3}  f\sqrt{\mathcal{F}_0} \phi(\mathcal{I}) \,d\mathcal{I}\,dp\right)^2-\sum_{i=1}^3\left(\int_0^\infty\int_{\mathbb{R}^3} p^i f\sqrt{\mathcal{F}_0} \phi(\mathcal{I}) \,d\mathcal{I}\,\frac{dp}{p^0}\right)^2 \right\}^{1/2}.
	\end{align*}
where we used the oddness of $\F_0$. Using the relation,
	$$
	\sqrt{1+y}=1+\frac y2-\frac{y^2}{2(2+y+2\sqrt{1+y})},
	$$
one can see that $n-1$ is linearized as	
	\begin{align}\label{n linear}
	n-1&=\int_0^\infty\int_{\mathbb{R}^3}  f\sqrt{\mathcal{F}_0} \phi(\mathcal{I}) \,d\mathcal{I}\,dp+N_n
	\end{align}
	where $N_n$ is the nonlinear part of $n$ defined by
	\begin{equation}\label{Nn}
\frac 12 \left(\int_0^\infty\int_{\mathbb{R}^3}  f\sqrt{\mathcal{F}_0} \phi(\mathcal{I}) \,d\mathcal{I}\,dp\right)^2-\frac 12\sum_{i=1}^3\left(\int_0^\infty\int_{\mathbb{R}^3} p^i f\sqrt{\mathcal{F}_0} \phi(\mathcal{I}) \,d\mathcal{I}\,\frac{dp}{p^0}\right)^2-\frac{\Phi^2}{2(2+\Phi+2\sqrt{1+\Phi})}
	\end{equation}	
	and $\Phi$ denotes
	$$
	\Phi:=2\int_0^\infty\int_{\mathbb{R}^3}  f\sqrt{\mathcal{F}_0} \phi(\mathcal{I}) \,d\mathcal{I}\,dp+ \left(\int_0^\infty\int_{\mathbb{R}^3}  f\sqrt{\mathcal{F}_0} \phi(\mathcal{I}) \,d\mathcal{I}\,dp\right)^2-\sum_{i=1}^3\left(\int_0^\infty\int_{\mathbb{R}^3} p^i f\sqrt{\mathcal{F}_0} \phi(\mathcal{I}) \,d\mathcal{I}\,\frac{dp}{p^0}\right)^2.
	$$
$\bullet$ (Linearization of $u$): Inserting \eqref{n linear} into the definition of $u$ gives
	\begin{align}\label{linear u}\begin{split}
	u &=\frac {1}{n}\int_0^\infty\int_{\mathbb{R}^3} p (\mathcal{F}_0+\sqrt{\mathcal{F}_0}f) \phi(\mathcal{I}) \,d\mathcal{I}\,\frac{dp}{p^0}\cr 
	&=\frac {1}{1+\int_0^\infty\int_{\mathbb{R}^3}  f\sqrt{\mathcal{F}_0} \phi(\mathcal{I}) \,d\mathcal{I}\,dp+N_n}\int_0^\infty\int_{\mathbb{R}^3} p  f\sqrt{\mathcal{F}_0}  \phi(\mathcal{I}) \,d\mathcal{I}\,\frac{dp}{p^0}\cr 
	&=\int_0^\infty\int_{\mathbb{R}^3} p  f\sqrt{\mathcal{F}_0}  \phi(\mathcal{I}) \,d\mathcal{I}\,\frac{dp}{p^0}+N_u
\end{split}	\end{align}
	where $N_u$ is the nonlinear part of $u$ defined by
	\begin{equation}\label{Nu}
	-\frac{\int_0^\infty\int_{\mathbb{R}^3}  f\sqrt{\mathcal{F}_0} \phi(\mathcal{I}) \,d\mathcal{I}\,dp+N_n}{1+\int_0^\infty\int_{\mathbb{R}^3}  f\sqrt{\mathcal{F}_0} \phi(\mathcal{I}) \,d\mathcal{I}\,dp+N_n} \int_0^\infty\int_{\mathbb{R}^3} p  f\sqrt{\mathcal{F}_0}  \phi(\mathcal{I}) \,d\mathcal{I}\,\frac{dp}{p^0}.
	\end{equation}
	$\bullet$ (Linearization of $\eta-\eta_0$): 
Using the relation
$$
\frac{1}{1+y}=1-y+\frac{y^2}{1+y},
$$
one finds
	\begin{align}\label{linear eta}\begin{split}
	\eta&=\frac{1}{n}\inti\intr  (\mathcal{F}_0+\sqrt{\mathcal{F}_0}f)(1+\I)^{-1} \phi(\I)\, d\I\frac{dp}{p^0}\cr 
	&=\frac {1}{1+\int_0^\infty\int_{\mathbb{R}^3}  f\sqrt{\mathcal{F}_0} \phi(\mathcal{I}) \,d\mathcal{I}\,dp+N_n} \left\{ \eta_0+\inti\intr  f\sqrt{\mathcal{F}_0} (1+\I)^{-1} \phi(\I)\, d\I\frac{dp}{p^0}\right\}\cr 
	&=\eta_0-\eta_0\int_0^\infty\int_{\mathbb{R}^3}  f\sqrt{\mathcal{F}_0} \phi(\mathcal{I}) \,d\mathcal{I}\,dp+   \inti\intr  f\sqrt{\mathcal{F}_0} (1+\I)^{-1} \phi(\I)\, d\I\frac{dp}{p^0} +N_\eta
	\end{split}\end{align}
where $N_\eta$ is the nonlinear part of $\eta$ defined by
\begin{equation}\label{Neta}
-\eta_0 N_n+\eta_0\frac{\Psi^2}{1+\Psi}-\frac{\Psi}{1+\Psi}\inti\intr  f\sqrt{\mathcal{F}_0} (1+\I)^{-1} \phi(\I)\, d\I\frac{dp}{p^0}
\end{equation}
and $\Psi$ denotes
$$
\Psi:=\int_0^\infty\int_{\mathbb{R}^3}  f\sqrt{\mathcal{F}_0} \phi(\mathcal{I}) \,d\mathcal{I}\,dp+N_n.
$$	
Finally, combining \eqref{theta2}, \eqref{n linear}, \eqref{linear u} and \eqref{linear eta} we conclude that
\begin{align*}
&\frac{1}{\sqrt{\F_0}}\left(\mathcal{F}_E-F\right)\cr 
&=\left\{\int_0^\infty\int_{\mathbb{R}^3}  f\sqrt{\mathcal{F}_0} \phi(\mathcal{I}) \,d\mathcal{I}\,dp\right\}\sqrt{\F_0}+\left(1+\I\right)\gamma_0  \left\{\int_0^\infty\int_{\mathbb{R}^3} p  f\sqrt{\mathcal{F}_0}  \phi(\mathcal{I}) \,d\mathcal{I}\,\frac{dp}{p^0}\right\}\cdot p \sqrt{\F_0}\cr 
&+\frac{M^2(\gamma_0)}{M^2(\gamma_0)+M'(\gamma_0)\widetilde{M}(\gamma_0)}\left\{  \inti\intr  \left(\frac{1}{(1+\I)p^0}-\eta_0\right)f\sqrt{\mathcal{F}_0}  \phi(\I)\, d\I dp\right\}\left(\frac{M^\prime(\gamma_0 )}{M(\gamma_0 )}+(1+\I)  p^0\right)\sqrt{\F_0}\cr 
&+N_n\sqrt{\F_0}+\left(1+\I\right)\gamma_0  N_u\cdot p \sqrt{\F_0}+\frac{M^2(\gamma_0)}{M^2(\gamma_0)+M'(\gamma_0)\widetilde{M}(\gamma_0)}\left(\frac{M^\prime(\gamma_0 )}{M(\gamma_0 )}+(1+\I)  p^0\right)  N_\eta\sqrt{\F_0}\cr 
&+\frac{1}{\sqrt{\F_0}}\int_0^1 (1-\theta) \frac{d^2}{d\theta^2}\left\{\mathcal{F}(n_\theta, u_\theta,\eta_\theta)\right\} \,d\theta,
\end{align*} 
which completes the proof.

\end{proof}

Let $\mathcal{N}$ be the five dimensional space spanned by $\{\sqrt{\F_0},(1+\I)p^\mu\sqrt{\F_0}\}$ and denote $e_i$ $(i=1,\cdots,5)$ by
$$
e_1=\sqrt{\F_0},\qquad  e_{1+i}=\frac{(1+\I)p^i\sqrt{\F_0}}{\|(1+\I)p^i\sqrt{\F_0}\|_{L^2_{p,\I}}},  \qquad e_5=\frac{\{(1+\I)p^0-\delta\}\sqrt{\F_0}}{\|\{(1+\I)p^0-\delta\}\sqrt{\F_0}\|_{L^2_{p,\I}}},
 $$
where $\delta$ is the constant defined by
$$
\delta=	\frac{\inti \intr (1+\I) p^0 \F_0\phi(\I)\, d\I dp}{\inti \intr    \F_0\phi(\I)\, d\I dp}. 
$$
By oddness and definition of $\delta$, $e_1,\cdots,e_5$ are mutually orthonormal with respect to the weighted inner product $\langle \cdot,\cdot\rangle_{L^2_{p,\I}}$. Then the following  operator 
\begin{align}\label{Pf}
\mathcal{P}(f)=\langle f,e_1\rangle_{L^2_{p,\I}} e_1+\langle f,e_{2,3,4} \rangle_{L^2_{p,\I}} \cdot e_{2,3,4}+ \langle f,e_{5}\rangle_{L^2_{p,\I}} e_5.
\end{align}
is an orthonormal projection from $L^2\big(\R^3\times [0,\infty)\big)$ onto  $\mathcal{N}$ with respect to the weighted inner product $\langle \cdot,\cdot\rangle_{L^2_{p,\I}}$. Let us define the linear operator $\mathcal{L}$ as
$$
\mathcal{L}(f):=\frac{1}{(1+\I)p^0}\{\mathcal{P}_0(f)-f\}.
$$
By Lemma \ref{linearization}, the relativistic BGK model \eqref{PMarle} can be rewritten by
\begin{equation}\label{linearized eqn}
\pa_tf  +\frac{p}{p^0}\cdot\nabla_x f =\mathcal{L}(f)+\Gamma(f).
\end{equation}
\begin{proposition} \label{linear property} The linearized operator $\mathcal{L}$ satisfies the followings.
	\begin{enumerate}
		\item $\mathcal{L}$ is self-adjoint with respect to $\langle \cdot,\cdot\rangle_{L^2_{p,\I}}$.
			\item $\displaystyle Ker(\mathcal{L})=\{\sqrt{\F_0},(1+\I)p^\mu\sqrt{\F_0}\}$.
		\item There exists a positive constant $\lambda$ such that
		$$
		\langle L(f),f\rangle_{L^2_{p,\I}} \le -\lambda \|(I-\mathcal{P})f\|^2_{L^2_{p,\I}}
		$$
	\end{enumerate}
\end{proposition}
\begin{proof} 
$\bullet$ Proof of (1): Rearranging the operator $\mathcal{P}_0$, we have
	\begin{align}\label{rearrange}\begin{split}
	&\left\{ (1+\I)p^0\right\}^{-1}\mathcal{P}_0(f)\cr 
	&=\frac{M^\prime(\gamma_0 ) \left\{ M(\gamma_0) +\eta_0  M^\prime(\gamma_0 ) \right\}}{M^2(\gamma_0)+M'(\gamma_0)\widetilde{M}(\gamma_0)} \left\{ (1+\I)p^0\right\}^{-1}\sqrt{\F_0}  \iint  \left\{ (1+\I)p^0\right\}^{-1}f\sqrt{\mathcal{F}_0}  \phi(\I)\, d\I dp \cr 
	&+\gamma_0  \frac{p}{p^0} \sqrt{\F_0}\cdot \iint p  f\sqrt{\mathcal{F}_0}  \phi(\mathcal{I}) \,d\mathcal{I}\,\frac{dp}{p^0}\cr 
	&-\frac{M(\gamma_0)\widetilde{M}(\gamma_0)}{M^2(\gamma_0)+M'(\gamma_0)\widetilde{M}(\gamma_0)}\left(1+\frac{M^\prime(\gamma_0 )}{M(\gamma_0 )(1+\I)p^0 } \right)\sqrt{\F_0}\iint  \left(1+\frac{M^\prime(\gamma_0 )}{M(\gamma_0 )(1+\I)p^0 }  \right)f\sqrt{\mathcal{F}_0}  \phi(\I)\, d\I dp\cr 
	&+\sqrt{\F_0} \iint  \{(1+\I)p^0\}^{-1}f\sqrt{\mathcal{F}_0}  \phi(\I)\, d\I dp+\{(1+\I)p^0\}^{-1 }\sqrt{\F_0}\iint  f\sqrt{\mathcal{F}_0} \phi(\mathcal{I}) \,d\mathcal{I}\,dp.
\end{split}	\end{align}
Thus one finds
\begin{align}\label{P_0 self-adjoint}
\left\langle \{(1+\I)p^0\}^{-1}\mathcal{P}_0(f),g \right\rangle_{L^2_{p,\I}}=\left\langle f,\{(1+\I)p^0\}^{-1}\mathcal{P}_0(g) \right\rangle_{L^2_{p,\I}},
\end{align}	
which completes the proof of (1).\noindent\newline
$\bullet$ Proof of (2): By definition of $\mathcal{L}$, it suffices to show that 
\begin{equation}\label{P_0}
\mathcal{P}_0\left(\sqrt{\mathcal{F}_0}\right)=\sqrt{\mathcal{F}_0},\quad \mathcal{P}_0\left((1+\I)p^\mu\sqrt{\mathcal{F}_0}\right)=(1+\I)p^\mu\sqrt{\mathcal{F}_0}.
\end{equation}
For the case of $\sqrt{\F_0}$, we see that
\begin{align*}
\mathcal{P}_0\left(\sqrt{\mathcal{F}_0}\right)&=\left\{\int_0^\infty\int_{\mathbb{R}^3}  \mathcal{F}_0 \phi(\mathcal{I}) \,d\mathcal{I}\,dp \right\}\sqrt{\mathcal{F}_0}\cr 
&+\frac{M(\gamma_0)M^\prime(\gamma_0)}{M^2(\gamma_0)+M'(\gamma_0)\widetilde{M}(\gamma_0)}\left\{ \inti\intr   \left(\frac{1}{(1+\I)p^0}-\frac{\widetilde{M}(\gamma_0)}{M(\gamma_0)} \right)\mathcal{F}_0  \phi(\I)\, d\I dp  \right\}\sqrt{\mathcal{F}_0}\cr 
&+\frac{M^2(\gamma_0)}{M^2(\gamma_0)+M'(\gamma_0)\widetilde{M}(\gamma_0)}   \left\{ \inti\intr   \left(\frac{1}{(1+\I)p^0}-\frac{\widetilde{M}(\gamma_0)}{M(\gamma_0)} \right)\mathcal{F}_0  \phi(\I)\, d\I dp  \right\}(1+\I)  p^0 \sqrt{\mathcal{F}_0}\cr 
&= \sqrt{\mathcal{F}_0}
\end{align*}
where we used \eqref{eta0} and the fact that $\F_0$ is radially symmetric. For $p^i\sqrt{\F_0}$, we observe that
	\begin{align*}
\mathcal{P}_0\left((1+\I)p^i\sqrt{\mathcal{F}_0}\right)&=\gamma_0\left(1+ \mathcal{I} \right) p^i\sqrt{\mathcal{F}_0}  \int_0^\infty\int_{\mathbb{R}^3} (1+\I)|p^i|^2  \mathcal{F}_0  \phi(\mathcal{I}) \,d\mathcal{I}\,\frac{dp}{p^0}\cr 
&=\frac {\gamma_0}{3M(\gamma_0)}\left(1+ \mathcal{I} \right) p^i\sqrt{\mathcal{F}_0}  \int_0^\infty\int_{\mathbb{R}^3} (1+\I)\left\{(p^0)^2-1 \right\}  e^{-(1+\I)\gamma_0p^0}  \phi(\mathcal{I}) \,d\mathcal{I}\,\frac{dp}{p^0}. 
\end{align*}
Applying spherical coordinates and integration by parts gives
	\begin{align*}
\int_0^\infty\int_{\mathbb{R}^3}  (1+\I)p^0e^{-(1+\I)\gamma_0p^0}  \phi(\mathcal{I}) \,d\mathcal{I}dp&=4\pi\inti \inti (1+\I)r^2\sqrt{1+r^2}e^{-(1+\I)\gamma_0\sqrt{1+r^2}}\phi (\I)\,d\I dr\cr 
&=\frac{4\pi}{\gamma_0} \inti \inti (3r^2+1)e^{-(1+\I)\gamma_0\sqrt{1+r^2}}\phi (\I)\,d\I dr\cr 
&=\frac{3M(\gamma_0)}{\gamma_0} +\frac{4\pi}{\gamma_0} \inti \inti e^{-(1+\I)\gamma_0\sqrt{1+r^2}}\phi (\I)\,d\I dr 
\end{align*}
and
 \begin{align*}
	 \int_0^\infty\int_{\mathbb{R}^3}  (1+\I)e^{-(1+\I)\gamma_0p^0}  \phi(\mathcal{I}) \,d\mathcal{I}\,\frac{dp}{p^0}&= 4\pi\inti \inti \frac{r^2}{\sqrt{1+r^2}} (1+\I)e^{-(1+\I)\gamma_0\sqrt{1+r^2}}\phi (\I)\,d\I dr\cr 
	&=\frac{4\pi}{\gamma_0 } \inti \inti e^{-(1+\I)\gamma_0\sqrt{1+r^2}}\phi (\I)\,d\I dr.
	\end{align*}
Thus we get
	\begin{align*}
\mathcal{P}_0\left((1+\I)p^i\sqrt{\mathcal{F}_0}\right)&= \left(1+ \mathcal{I} \right) p^i\sqrt{\mathcal{F}_0}.
\end{align*}
Similarly, 
	\begin{align*}
	&\mathcal{P}\left((1+\I)p^0\sqrt{\mathcal{F}_0}\right)\cr 
	&=\left\{\int_0^\infty\int_{\mathbb{R}^3} (1+\I)p^0 \mathcal{F}_0 \phi(\mathcal{I}) \,d\mathcal{I}\,dp \right\}\sqrt{\mathcal{F}_0}\cr 
	&+\frac{M(\gamma_0)M^\prime(\gamma_0)}{M^2(\gamma_0)+M'(\gamma_0)\widetilde{M}(\gamma_0)}  \left\{ \inti\intr    \mathcal{F}_0    \phi(\I)\, d\I dp -\frac{\widetilde{M}(\gamma_0)}{M(\gamma_0)}\int_0^\infty\int_{\mathbb{R}^3} (1+\I)p^0\mathcal{F}_0 \phi(\mathcal{I}) \,d\mathcal{I}\,dp \right\}\sqrt{\mathcal{F}_0}\cr 
	&+\frac{M^2(\gamma_0)}{M^2(\gamma_0)+M'(\gamma_0)\widetilde{M}(\gamma_0)}   \left\{ \inti\intr    \mathcal{F}_0  \phi(\I)\, d\I dp -\frac{\widetilde{M}(\gamma_0)}{M(\gamma_0)}\int_0^\infty\int_{\mathbb{R}^3} (1+\I)p^0 \mathcal{F}_0 \phi(\mathcal{I}) \,d\mathcal{I}\,dp \right\}(1+\I)  p^0 \sqrt{\mathcal{F}_0}\cr 
	&=-\frac{M'(\gamma_0)}{M(\gamma_0)}\sqrt{\mathcal{F}_0}+\frac{M(\gamma_0)M^\prime(\gamma_0)}{M^2(\gamma_0)+M'(\gamma_0)\widetilde{M}(\gamma_0)}  \left\{ 1 +\frac{\widetilde{M}(\gamma_0)M'(\gamma_0)}{M^2(\gamma_0)} \right\}\sqrt{\mathcal{F}_0}\cr 
	&+\frac{M^2(\gamma_0)}{M^2(\gamma_0)+M'(\gamma_0)\widetilde{M}(\gamma_0)} \left\{ 1 +\frac{\widetilde{M}(\gamma_0)M'(\gamma_0)}{M^2(\gamma_0)} \right\}(1+\I)  p^0 \sqrt{\mathcal{F}_0}\cr 
	&=(1+\I)  p^0 \sqrt{\mathcal{F}_0}.
	\end{align*}
$\bullet$ Proof of (3): 
Since we have proved (1) and (2), it is sufficient to deal with functions that are not identically zero and belong to $\mathcal{N}^\perp$. Indeed, for any $f=g+h$ with $g\in \mathcal{N}$ and $h\in \mathcal{N}^\perp$, 
$$
\langle L(f),f\rangle_{L^2_{p,\I}} =\langle L(h),g+h\rangle_{L^2_{p,\I}}=\langle L(h),h\rangle_{L^2_{p,\I}},\qquad \|(I-\mathcal{P})f\|^2_{L^2_{p,\I}}=\|(I-\mathcal{P})h\|^2_{L^2_{p,\I}}.
$$
Suppose that for any $\epsilon>0$, there exists a function $g_\epsilon(p,\I)$ such that 
$$
\inti\int_{\mathbb{R}^3} g_\epsilon\sqrt{\F_0}\phi(\I)\,d\I dp=0,\quad \inti\int_{\mathbb{R}^3} 
(1+\I)p^\mu
 g_\epsilon\sqrt{\F_0}\phi(\I)\,d\I dp=0
$$
and 
$$
\left\langle \mathcal{L}(g_\epsilon),g_\epsilon\right\rangle_{L^2_{p,\I}} > -\epsilon\|(I-\mathcal{P})g_\epsilon\|^2_{L^2_{p,\I}}.
$$
Since $\|(I-\mathcal{P})g_\epsilon\|_{L^2_{p,\I}}$ is strictly positive, we consider the normalized functions $h_n(p,\I)$ such that for any $n\in\mathbb{N}$, 
\begin{equation}\label{perp}
\inti\int_{\mathbb{R}^3}h_n\sqrt{\F_0}\phi(\I)\,d\I dp=0,\quad \inti\int_{\mathbb{R}^3}(1+\I)p^\mu h_n\sqrt{\F_0}\phi(\I)\,d\I dp=0, 
\end{equation}
and 
\begin{align}\label{contrary}\begin{split}
\left \langle -\mathcal{L}(h_n),h_n\right\rangle_{L^2_{p,\I}}&=\left\langle h_n, h_n \right\rangle_{L^2_w}-\left\langle \mathcal{P}_0(h_n), h_n \right\rangle_{L^2_w}\cr 
&=1-\left\langle \mathcal{P}_0(h_n), h_n \right\rangle_{L^2_w}\cr 
&<\frac1n 
\end{split}\end{align}
where
$$
\langle f,g\rangle_{L^2_w}:=\left\langle \{(1+\I)p^0\}^{-1}f, g \right\rangle_{L^2_{p,\I}}
$$
Let $h_0$ be a weak limit of $\{h_n\}$ with respect to $\langle \cdot,\cdot \rangle_{L^2_w}$ up to a subsequence with $\left\langle h_0, h_0 \right\rangle_{L^2_w}\le  1$. Since the operator $\mathcal{P}_0$ is bounded and has a finite rank, $\{(1+\I)p^0\}^{-1}\mathcal{P}_0$ is compact in $L^2_{p,\I}$. Thus we have
$$
\lim_{n\to \infty}\langle \mathcal{P}_0(h_n),h_n\rangle_{L^2_w}=\langle \mathcal{P}_0(h_0),h_0\rangle_{L^2_w}.
$$
Also, we notice that
\begin{align}\label{nonnegative}\begin{split}
	\left\langle -\mathcal{L}(h_n),h_n\right\rangle_{L^2_{p,\I}}&=	\left\langle h_n-\mathcal{P}_0(h_n),h_n-\mathcal{P}_0(h_n)\right\rangle_{L^2_w}+\left\langle h_n,\mathcal{P}_0(h_n)\right\rangle_{L^2_w}-\left\langle \mathcal{P}_0(h_n),\mathcal{P}_0(h_n)\right\rangle_{L^2_w}\cr
	&=	\left\langle h_n-\mathcal{P}_0(h_n),h_n-\mathcal{P}_0(h_n)\right\rangle_{L^2_w}+\left\langle h_n,\mathcal{P}_0(h_n)\right\rangle_{L^2_w}-\left\langle h_n,\mathcal{P}_0(h_n)\right\rangle_{L^2_w}\cr
	&\ge 0 
\end{split}\end{align}
due to \eqref{P_0 self-adjoint} and \eqref{P_0}. Thus we have from  \eqref{contrary} that
$$
0=\lim_{n\to \infty}\left\langle -\mathcal{L}(h_n),h_n\right\rangle_{L^2_w}=1-\langle \mathcal{P}(h_0),h_0\rangle_{L^2_w}=1-\langle h_0,h_0\rangle_{L^2_w}+\langle -\mathcal{L}(h_0),h_0\rangle_{L^2_{p,\I}}\ge \langle -\mathcal{L}(h_0),h_0\rangle_{L^2_{p,\I}}.
$$
Since \eqref{nonnegative} also holds for $h_0$, we conclude that \begin{equation}\label{conclusion}
\left\langle \mathcal{L}(h_0),h_0\right\rangle_{L^2_{p,\I}}=0,\quad\ \langle h_0, h_0 \rangle_{L^2_w}=1.
\end{equation}
On the other hand, taking the limit $n\to\infty$ to \eqref{perp}, we obtain $h_0\in Ker(L)^\perp$. However \eqref{conclusion} implies that $h_0$ is not identically zero and belongs to $Ker(L)$, a contradiction. This completes the proof.
\end{proof}

\section{Technical lemmas}
The aim of this section is to establish the local-in-time existence and uniqueness of solutions of \eqref{PMarle}. For this, we first present some estimates in order to control the nonlinear term $\Gamma$.
%%%%%%%%%%%%%%%%%%%%%%%%%%%%%%%%%%%%%%%%%%%%%%%%%%%%%%%%%%%%%%%%%%%%%%%%%%%%%%%%%%%%%%%%%%%%%%%%%%%%%%%%%%%%%%%%%%%%%%%%%%%%%%
%
%
%
%
%%%%%%%%%%%%%%%%%%%%%%%%%%%%%%%%%%%%%%%%%%%%%%%%%%%%%%%%%%%%%%%%%%%%%%%%%%%%%%%%%%%%%%%%%%%%%%%%%%%%%%%%%%%%%%%%%%%%%%%%%%%%%%%
\begin{lemma}\label{lem1}
Let $N\ge 3$.	Suppose that $\mathcal{E}(f)(t)$ is sufficiently small. We  then have
	$$
	\sum_{|\alpha|\le N}\left( |\pa^\alpha\{n-1\}|+|\pa^\alpha u|+|\pa^\alpha\{\eta-\eta_0\}|\right)\le C\sum_{|\alpha|\le N}\|\pa^\alpha f\|_{L^2_{p,\I}} \le C \sqrt{\mathcal{E}(f)(t)}.
	$$
\end{lemma}
\begin{proof}
%Recall from \eqref{n,u} that
%	\begin{align*}
%	n=\left\{\left(1+\int_{\mathbb{R}^3}\int_0^\infty  f\sqrt{\mathcal{F}_0} \phi(\mathcal{I}) \,d\mathcal{I}\,dp\right)^2-\sum_{i=1}^3\left(\int_{\mathbb{R}^3}\int_0^\infty p^i f\sqrt{\mathcal{F}_0} \phi(\mathcal{I}) \,d\mathcal{I}\,\frac{dp}{p^0}\right)^2 \right\}^{1/2}.
%	\end{align*}
Using the Sobolev embedding $H^2(\R^3)\subseteq L^\infty(\R^3)$, we get
	\begin{align}\label{Sobolev}
	\left| \int_{\mathbb{R}^3}\int_0^\infty p^0 f\sqrt{\mathcal{F}_0} \phi(\mathcal{I}) \,d\mathcal{I}\,dp\right|\le \left\| \|f\|_{L^2_{p,\I}} \right\|_{L^\infty(\T_x)} \left(\int_{\mathbb{R}^3}\int_0^\infty p^0 \mathcal{F}_0 \phi(\mathcal{I}) \,d\mathcal{I}\,dp\right)^{\frac 12}\le C \sqrt{\mathcal{E}(f)(t)}.
	\end{align}
Recalling \eqref{n linear}, \eqref{linear u} and \eqref{linear eta},
it is easy to check that $|\pa^\alpha\{n-1\}|,$ $|\pa^\alpha u|$ and $|\pa^\alpha\{\eta-\eta_0\}|$ take the form of rational functions whose denominators are dominated by positive constants due to \eqref{Sobolev}. For instance,
\begin{align*} 
\pa^\alpha\{	n-1\}&=\int_0^\infty\int_{\mathbb{R}^3}  \pa^\alpha f\sqrt{\mathcal{F}_0} \phi(\mathcal{I}) \,d\mathcal{I}\,dp+\pa^\alpha N_n
\end{align*}
where $N_n$ denotes
\begin{align*} 
N_n&=\frac 12 \left(\int_0^\infty\int_{\mathbb{R}^3}  f\sqrt{\mathcal{F}_0} \phi(\mathcal{I}) \,d\mathcal{I}\,dp\right)^2-\frac 12\sum_{i=1}^3\left(\int_0^\infty\int_{\mathbb{R}^3} p^i f\sqrt{\mathcal{F}_0} \phi(\mathcal{I}) \,d\mathcal{I}\,\frac{dp}{p^0}\right)^2-\frac{\Phi^2}{2(2+\Phi+2\sqrt{1+\Phi})}
\end{align*}	
and $\Phi$ is given by
$$
\Phi=2\int_0^\infty\int_{\mathbb{R}^3}  f\sqrt{\mathcal{F}_0} \phi(\mathcal{I}) \,d\mathcal{I}\,dp+ \left(\int_0^\infty\int_{\mathbb{R}^3}  f\sqrt{\mathcal{F}_0} \phi(\mathcal{I}) \,d\mathcal{I}\,dp\right)^2-\sum_{i=1}^3\left(\int_0^\infty\int_{\mathbb{R}^3} p^i f\sqrt{\mathcal{F}_0} \phi(\mathcal{I}) \,d\mathcal{I}\,\frac{dp}{p^0}\right)^2.
$$
Applying \eqref{Sobolev}, the denominator can be estimated as 
$$
\frac{1}{\left(2+\Phi+2\sqrt{1+\Phi}\right)^{2^{|\alpha|}}}\le \frac{1}{4-C\sqrt{\mathcal{E}(f)(t)}+4\sqrt{1-C\sqrt{\mathcal{E}(f)(t)}}} \le C
$$
 for sufficiently small $\mathcal{E}(f)(t)$. Thus, in the case of $\alpha=0$, we deduce that
$$
|n-1| \le C \sqrt{\mathcal{E}(f)(t)}.
$$
For the case of $\alpha\neq 0$,  we observe that
\begin{align*}
&\pa^\alpha \left(\int_0^\infty\int_{\mathbb{R}^3} p^0 f\sqrt{\mathcal{F}_0} \phi(\mathcal{I}) \,d\mathcal{I}\,dp\right)^2\cr 
&=\sum_{\substack{|\alpha'|\le |\alpha|\cr |\alpha'|\le |\alpha-\alpha'|}} C_{\alpha'}\int_0^\infty\int_{\mathbb{R}^3}p^0 \pa^{\alpha'} f\sqrt{\mathcal{F}_0} \phi(\mathcal{I}) \,d\mathcal{I}\,dp\int_0^\infty\int_{\mathbb{R}^3}p^0 \pa^{\alpha-\alpha'}  f\sqrt{\mathcal{F}_0} \phi(\mathcal{I}) \,d\mathcal{I}\,dp\cr 
&+\sum_{\substack{|\alpha'|\le |\alpha|\cr |\alpha-\alpha'|< |\alpha'|}} C_{\alpha'}\int_0^\infty\int_{\mathbb{R}^3}p^0 \pa^{\alpha'} f\sqrt{\mathcal{F}_0} \phi(\mathcal{I}) \,d\mathcal{I}\,dp\int_0^\infty\int_{\mathbb{R}^3}p^0 \pa^{\alpha-\alpha'}  f\sqrt{\mathcal{F}_0} \phi(\mathcal{I}) \,d\mathcal{I}\,dp\cr 
& \le C\sqrt{\mathcal{E}(f)(t)}\sum_{\substack{|\alpha'|\le |\alpha|\cr |\alpha'|\le |\alpha-\alpha'|}}   \|\pa^{\alpha-\alpha'}f\|_{L^2_{p,\I}}+C\sqrt{\mathcal{E}(f)(t)}\sum_{\substack{|\alpha'|\le |\alpha|\cr |\alpha-\alpha'|< |\alpha'|}}\|\pa^{\alpha^\prime} f\|_{L^2_{p,\I}}\cr 
&\le C\sqrt{\mathcal{E}(f)(t)}\sum_{|\alpha'|\le|\alpha|}   \|\pa^{\alpha'}f\|_{L^2_{p,\I}}.
\end{align*}
Here we can use \eqref{Sobolev} thanks to $N\ge 3$.  Similarly,
\begin{align*}
\pa^\alpha \left(\int_0^\infty\int_{\mathbb{R}^3}  f\sqrt{\mathcal{F}_0} \phi(\mathcal{I}) \,d\mathcal{I}\,dp\right)^n&\le C_n\left\{\mathcal{E}(f)(t)\right\}^{\frac {n-1}2}\sum_{|\alpha'|\le|\alpha|}   \|\pa^{\alpha'}f\|_{L^2_{p,\I}}
\end{align*}
for $n\in \mathbb{N}$. From these observations, we  conclude that
	\begin{align*} 
|\pa^\alpha\{	n-1\}|&\le \|\pa^\alpha f\|_{L^2_{p,\I}}+C\sqrt{\mathcal{E}(f)(t)}\sum_{|\alpha|\le N}   \|\pa^{\alpha}f\|_{L^2_{p,\I}}. 
\end{align*}
In the almost same manner, one finds
	\begin{align*} 
|\pa^\alpha	u|+|\pa^\alpha\{	\eta-\eta_0\}|&\le \|\pa^\alpha f\|_{L^2_{p,\I}}+C\sqrt{\mathcal{E}(f)(t)}\sum_{|\alpha|\le N}   \|\pa^{\alpha}f\|_{L^2_{p,\I}},
\end{align*}
which gives the desired result.

\end{proof}

\begin{remark}
	Recall that the transitional variables are defined as
	$$
	(n_\theta, u_\theta,\eta_\theta;p)=\theta(n-1, u,\eta-\eta_0)+(1,0,\eta_0).
	$$
	Thus it follows from Lemma \ref{lem1} that
	\begin{align}\label{theta estimate}
	|\pa^\alpha\{n_\theta-1\}|+|\pa^\alpha u_\theta|+|\pa^\alpha\{\eta_\theta-\eta_0\}|\le C \sqrt{\mathcal{E}(f)(t)}.
	\end{align}
\end{remark}
%%%%%%%%%%%%%%%%%%%%%%%%%%%%%%%%%%%%%%%%%%%%%%%%%%%%%%%%%%%%%%%%%%%%%%%%%%%%%%%%%%%%%%%%%%%%%%%%%%%%%%%%%%%%%%%%%%%%%%%%%%%%%%
%
%
%
%
%%%%%%%%%%%%%%%%%%%%%%%%%%%%%%%%%%%%%%%%%%%%%%%%%%%%%%%%%%%%%%%%%%%%%%%%%%%%%%%%%%%%%%%%%%%%%%%%%%%%%%%%%%%%%%%%%%%%%%%%%%%%%%%

\begin{lemma}\label{Appendix}
The function $(\widetilde{M}/M)(\cdot)$ is increasing in $\gamma\in (0,\infty)$
\end{lemma}
\begin{proof}
	It is straightforward that $ \widetilde{M}^\prime(\gamma) =-M(\gamma)$ and
	\begin{align*}
	M^\prime(\gamma) =-\inti\intr  p^0(1+\I)e^{-\left(1+ \mathcal{I} \right) \gamma p^0 } \phi(\mathcal{I}) \,d\mathcal{I}  dp .
	\end{align*}
	Thus, it follows from the H\"{o}lder inequality that
	\begin{align}\label{huu}
	\left(\frac{\widetilde{M}(\gamma)}{M(\gamma)} \right)^{\prime}&=\frac{ \widetilde{M}^\prime(\gamma)  M(\gamma)-  M^\prime (\gamma)  \widetilde{M}(\gamma) }{\{M(\gamma)\}^2} >0
	\end{align}
	for $\gamma\in(0,\infty)$, which completes the proof.
\end{proof}

\begin{lemma}\label{lem5}
Under the same assumption in Lemma \ref{lem1}, there exist positive constants $C$ and $C_n$ such that
	$$
	C^{-1}\le \big| \mathcal{X}(\eta_{\theta})\big|\leq C,\quad \Big|\frac{d^n}{d\eta_\theta^n} \mathcal{X}(\eta_{\theta})\Big|\le C_n
	$$
\end{lemma}
\begin{proof}

	By lemma \ref{Appendix}, we see that  $\mathcal{X}(\alpha)=(\widetilde{M}/M)^{-1}(\alpha)$ is an increasing function.
	Therefore, we have for small $\mathcal{E}(f)(t)$
	\begin{align}\label{lb1}
	\mathcal{X}(1/2\eta_0)\leq \mathcal{X}\big(\eta_0 -C\sqrt{\mathcal{E}(f)(t)}~\big)\leq\mathcal{X}(\eta_{\theta})\leq \mathcal{X}\big(\eta_0 +C\sqrt{\mathcal{E}(f)(t)}~\big)\leq \mathcal{X}(2\eta_0).
	\end{align}
	Note from \eqref{huu} that $(\widetilde{M}/M)^{\prime}$ is positive and is a continuous function on $\gamma\in (0,\infty)$ due to the definition of $M$ and $\widetilde{M}$, so  $(\widetilde{M}/M)^{\prime}$ has a positive lower bound and upper bound on the finite closed interval. Thus, due to \eqref{lb1} there exists a positive constant $C $ such that  
	\begin{align}\label{calX}
	C^{-1}\le 	\left|\mathcal{X}^{\prime}(\eta_{\theta})\right|&=\left|\left(\widetilde{M}/M\right)^{\prime}(\mathcal{X}(\eta_{\theta}))\right|^{-1}\le C.
	\end{align}
	On the other hand, letting $g(\mathcal{X}(\eta_\theta)):=(\widetilde{M}/M)^{\prime}(\mathcal{X}(\eta_\theta))$ then we see that
	$$
	\mathcal{X}^{(n)}(\eta_{\theta})=\left\{g(\mathcal{X}(\eta_{\theta})) \right\}^{-2^{n-1}} \mathbb{P}\left(g(\mathcal{X}(\eta_{\theta})),\cdots, g^{(n-1)}(\mathcal{X}(\eta_{\theta})),\mathcal{X}^\prime(\eta_{\theta}),\cdots,\mathcal{X}^{(n-1)}(\eta_{\theta}) \right)
	$$
	for some generic polynomial $\mathbb{P}$, where $g^{(k)}(\mathcal{X}(\eta_\theta))$  is bounded from above due to \eqref{lb1}. Thus by an induction over $n$ and \eqref{calX}, we conclude that there exists a positive constant $C_n$ such that
	$$
	\left| \mathcal{X}^{(n)}(\eta_{\theta})\right| \le C_n
	$$
\end{proof}

\begin{lemma}\label{lem55}
Under the same assumption in Lemma \ref{lem1}, there exist positive constants $C_{1}$ and $C_2$ depending on $n\in \N\cup \{0\}$
	such that
	\[
C_1^{-1}	\le \left| M^{(n)}(\mathcal{X}(\eta_{\theta}))\right|\leq C_{1},\quad C_2^{-1}	\le \left| \widetilde{M}^{(n)}(\mathcal{X}(\eta_{\theta}))\right|\leq C_{2}
	\]
\end{lemma}
\begin{proof}
	By definition of $M$, we see that 
	$$
	 M^{(n)}(\mathcal{X}(\eta_{\theta})) =\inti\intr  \{- (1+\I)p^0\}^n  e^{-(1+\I) \mathcal{X}(\eta_{\theta})p^0}\phi(\I)\,d\I dp.
	$$
Applying Lemma \ref{lem5}, we conclude that there exists a positive constant $C_1$ such that
	\begin{align*}
	\left|M^{(n)}(\mathcal{X}(\eta_{\theta}))	\right|	&\le   \intr \inti  \{(1+\I) p^0\}^n  e^{-C^{-1} (1+\I)p^0}\phi(\I)\,d\I dp\le C_1
	\end{align*}
	and
		\begin{align*}
	\left|M^{(n)}(\mathcal{X}(\eta_{\theta}))	\right|	&\ge   \intr \inti  \{(1+\I) p^0\}^n  e^{-C (1+\I)p^0}\phi(\I)\,d\I dp\ge C_1^{-1}.
	\end{align*}
In the same manner, there exists a positive constant $C_2$ such that
$$
C_2^{-1}\le\left| \widetilde{M}^{(n)}(\mathcal{X}(\eta_{\theta})) \right|=\inti\intr  \{ (1+\I)p^0\}^{n-1}  e^{-(1+\I) \mathcal{X}(\eta_{\theta})p^0}\phi(\I)\,d\I dp\le C_2.
$$
\end{proof}
%%%%%%%%%%%%%%%%%%%%%%%%%%%%%%%%%%%%%%%%%%%%%%%%%%%%%%%%%%%%%%%%%%%%%%%%%%%%%%%%%%%%%%%%%%%%%%%%%%%%%%%%%%%%%%%%%%%%%%%%%%%%%%%%%%%%%%%%%%%%%%%%%%%%%%%%%%%%%%%
%
%
%
%
%%%%%%%%%%%%%%%%%%%%%%%%%%%%%%%%%%%%%%%%%%%%%%%%%%%%%%%%%%%%%%%%%%%%%%%%%%%%%%%%%%%%%%%%%%%%%%%%%%%%%%%%%%%%%%%%%%%%%%%%%%%%%%%%%%%%%%%%%%%%%%%%%%%%%%%%%%%%%%

%\begin{lemma}\label{nonlinear3}
%	Suppose that $f$ is a solution to the perturbed relativistic BGK model \eqref{L-RBGK}. Then it holds that
%	$$
%	\inti\iint_{\R^3\times\R^3} \Gamma(f)\mathcal{P}(f) \phi(\I)\,d\I dpdx=0.
%	$$
%\end{lemma}
%\begin{proof}
%	Recall that
%	\begin{align*}
%	\frac{1}{(1+\I)p^{0}}(\F_{E}-F)=\sqrt{\F_0}\big\{ \mathcal{L}(f)+\Gamma(f) \big\},
%	\end{align*}
%which implies that
%	\begin{align*}
%	&\inti\iint_{\R^3\times\R^3} \Gamma(f)\mathcal{P}(f) \phi(\I)\,d\I dpdx\cr
%	%&=\iintPf\partial^{\alpha}\biggl\{-Lf+\frac{1}{p^{0}\sqrt{J_0}}\big(J_{F}-F\big)\biggl\}dxdq\cr
%	&=-\sum_{i=1}^5\langle f,e_i\rangle_{L^2_{p,\I}}\langle e_i,\mathcal{L}(f) \rangle_{L^2_{x,p,\I}}+\into\sum_{i=1}^5\langle f,e_i\rangle_{L^2_{p,\I}}\left\{\inti\intr \frac{e_i}{\sqrt{\F_0}}\frac{1}{(1+\I)p^0}\big(\mathcal{F}_E-F\big)\,d\I dp\right\}dx
%	\end{align*}
%	where
%	$$
%	e_1=\sqrt{\F_0},\qquad  e_{1+i}=\frac{(1+\I)p^i\sqrt{\F_0}}{\|(1+\I)p^i\sqrt{\F_0}\|_{L^2_{p,\I}}},  \qquad e_5=\frac{(1+\I)\left(p^0-\delta\right)\sqrt{\F_0}}{\|(1+\I)(p^0-\delta)\sqrt{\F_0}\|_{L^2_{p,\I}}}.
%	$$
%	This, combined with \eqref{conservation 1} and Lemma \ref{}, gives the desired result.	
%\end{proof}

\begin{lemma}\label{nonlinear1}
	Under the same assumption in Lemma \ref{lem1}, we have
	%\begin{align*}
	%\int_{\mathbb{R}^3}\partial^{\alpha}_{\beta}\Gamma(f)\phi dq&\le C\sqrt{\mathcal{D}(t)}\sum_{|k|\le|\alpha|}\|\partial^{k}f\|_{p^{0}}\|\phi\|_{p^{0}},
	%\end{align*}and
	$$
	\sum_{|\alpha|\le N}\left|\left\langle  \partial^{\alpha}\Gamma(f),g\right\rangle_{L^2_{p,\I}}\right|\le C\sqrt{\mathcal{E}(f)(t)}\sum_{|\alpha|\le N}\|\partial^{\alpha}f\|_{L^2_{p,\I}}\|g\|_{L^2_{p,\I}}.
	$$
\end{lemma}
\begin{proof}
	
	Recall from \eqref{Gamma} that
	\begin{align*}
	\Gamma(f)&=\frac{N_n}{(1+\I)p^0}\sqrt{\F_0}+\gamma_0  N_u\cdot \frac{p}{p^0} \sqrt{\F_0}+\frac{M^2(\gamma_0)}{M^2(\gamma_0)+M'(\gamma_0)\widetilde{M}(\gamma_0)}\left(\frac{1}{(1+\I)p^0}\frac{M^\prime(\gamma_0 )}{M(\gamma_0 )}+1\right)  N_\eta\sqrt{\F_0}\cr 
	&+\frac{1}{(1+\I)p^0\sqrt{\F_0}}\int_0^1 (1-\theta) (n-1,u,\eta-\eta_0)\nabla^2_{(n_\theta, u_\theta,\eta_\theta)}\mathcal{F}_\theta(n-1,u,\eta-\eta_0)^T \,d\theta,
	\end{align*}
	where the first three terms can be handled in the almost same manner as in the proof of Lemma \ref{lem1}, so we only deal with the last term to avoid the repetition.   Notice from Lemma \eqref{lem1} that
	$$
\sum_{|\alpha|\le N}\left( |\pa^\alpha\{n-1\}|+|\pa^\alpha u|+|\pa^\alpha\{\eta-\eta_0\}|\right)\le C\sum_{|\alpha|\le N}\|\pa^\alpha f\|_{L^2_{p,\I}} \le C \sqrt{\mathcal{E}(f)(t)},
$$
so it is enough to show that 
$$
\left| \inti \intr \frac{1}{(1+\I)p^0\sqrt{\F_0}}\left|\pa^\alpha\nabla^2_{(n_\theta, u_\theta,\eta_\theta)}\mathcal{F}_\theta\right| g(p,\I) \phi(\I)\,d\I dp \right|\le \|g\|_{L^2_{p,\I}}.
$$
It is straightforward that
	$$
	\nabla_{(n_{\theta},u_{\theta},\eta_{\theta})}^{2}\F_{\theta}=\left[ \mathcal{Q}_{i,j}\right]_{1\le i,j\le 5}\F_{\theta},
	$$
where the Hessian matrix $\left[ \mathcal{Q}_{i,j}\right]$ is symmetric defined by
\begin{align*} 
\mathcal{Q}_{1,1}&=0,\cr
\mathcal{Q}_{1,1+i}&=-(1+\I)\frac{\mathcal{X}(\eta_{\theta})}{n_{\theta}}\left(\frac{p^{0}}{u^{0}_{\theta}}u_{\theta}^{i}-p^{i}\right),\cr
\mathcal{Q}_{1,5}&=\frac{1}{n_{\theta}}\left\{\frac{M^2}{M^2+M'\widetilde{M}}\right\}\left(\mathcal{X}(\eta_{\theta})\right)\biggl(\left\{\frac{M^{\prime} }{M}\right\}\left(\mathcal{X}(\eta_{\theta})\right)+(1+\I)u_{\theta}^{\mu}p_{\mu}\biggl),\cr
\mathcal{Q}_{1+i,1+i}&=(1+\I)^2\mathcal{X}^{2}(\eta_{\theta})\left(\frac{p^{0}}{u^{0}_{\theta}}u_{\theta}^{i}-p^{i}\right)^{2}+(1+\I)\frac{\mathcal{X}(\eta_{\theta})}{\left(1+|u_{\theta}|^{2}\right)^{\frac 32}}\left\{(u_{\theta}^{i})^{2}-(1+|u_{\theta}|^{2})\right\}p^{0},\cr
\mathcal{Q}_{1+i,1+j}&=(1+\I)^2\mathcal{X}^{2}(\eta_{\theta})\left(p^{i}-\frac{p^{0}}{u^{0}_{\theta}}u_{\theta}^{i}\right)\left(p^{j}-\frac{p^{0}}{u^{0}_{\theta}}u^j_{\theta}\right)+(1+\I)\frac{\mathcal{X}(\eta_{\theta})}{\left(1+|u_{\theta}|^{2}\right)^{\frac 32}}u_{\theta}^{i}u_{\theta}^{j}p^{0}\quad (i\neq j),\cr
\mathcal{Q}_{1+i,5}&=(1+\I)\left\{\frac{M^2}{M^2+M'\widetilde{M}}\right\}\left(\mathcal{X}(\eta_{\theta})\right)\left(\frac{p^{0}}{u^{0}_{\theta}}u_{\theta}^{i}-p^i\right)\cr 
&-(1+\I)\mathcal{X}(\eta_{\theta})\left(\frac{p^{0}}{u^{0}_{\theta}}u_{\theta}^{i}-p^{i}\right)\left\{\frac{M^2}{M^2+M'\widetilde{M}}\right\}\left(\mathcal{X}(\eta_{\theta})\right)\biggl(\left\{\frac{M^{\prime} }{M}\right\}\left(\mathcal{X}(\eta_{\theta})\right)+(1+\I)u_{\theta}^{\mu}p_{\mu}\biggl),\cr
\mathcal{Q}_{5,5}&=\left\{\frac{M^2}{M^2+M'\widetilde{M}}\right\}^2\left(\mathcal{X}(\eta_{\theta})\right)\biggl(\left\{\frac{M^{\prime} }{M}\right\}\left(\mathcal{X}(\eta_{\theta})\right)+(1+\I)u_{\theta}^{\mu}p_{\mu}\biggl)^2\cr
&-\left\{\frac{M^2}{M^2+M'\widetilde{M}}\right\}^{\prime}\left(\mathcal{X}(\eta_{\theta})\right)\left\{\frac{M^2}{M^2+M'\widetilde{M}}\right\}\left(\mathcal{X}(\eta_{\theta})\right)\biggl(\left\{\frac{M^{\prime} }{M}\right\}\left(\mathcal{X}(\eta_{\theta})\right)+(1+\I)u_{\theta}^{\mu}p_{\mu}\biggl)\cr  
&-\left\{\frac{M^2}{M^2+M'\widetilde{M}}\right\}^2\left(\mathcal{X}(\eta_{\theta})\right) \left\{\frac{M^{\prime} }{M}\right\}^{\prime}\left(\mathcal{X}(\eta_{\theta})\right) 
\end{align*}
Notice that $M^2+M^\prime \widetilde{M}$ is strictly negative:
\begin{align*}
&(M^2+M^\prime \widetilde{M})(\gamma)\cr 
&=\left(\iint e^{-\gamma(1+\I) p^0}\phi(\I)\,d\I dp \right)^2-\iint (1+\I)p^0e^{-\gamma(1+\I) p^0}\phi(\I)\,d\I dp\iint({1+\I})^{-1} e^{-\gamma(1+\I) p^0}\phi(\I)\,d\I \frac{dp}{p^0}\cr 
&< 0
\end{align*}
where we used the H\"{o}lder inequality. Since it is a continuous function, we have from Lemma \ref{lem5} that 
$$
\left|\frac{1}{(M^2+M^\prime \widetilde{M})(\mathcal{X}(\eta_\theta))}\right|\le C.
$$
This, combined with Lemma \ref{lem5} and Lemma \ref{lem55} implies that all the terms consisting of  $M^{(n)}$, $\widetilde{M}^{(n)}$ and $\mathcal{X}$ are bounded from above. On the other hand, it follows from \eqref{theta estimate} that
$$
n_\theta\approx 1, \quad u_\theta\approx0,\quad u_\theta^0\approx 1
$$
for sufficiently small $\mathcal{E}(f)(t)$. Therefore we can conclude that for  $1\le i,j\le 5$,
$$
\left| \mathcal{Q}_{i,j}\right|\le C\left\{1+(1+\I)p^0   \right\}^2
$$
and similarly, 
\begin{equation}\label{theta1}
\left| \pa^\alpha\mathcal{Q}_{i,j}\right|\le C\left\{1+(1+\I)p^0   \right\}^2
\end{equation}
Next, we observe from \eqref{theta estimate} that
\begin{align*}
\F_\theta&=\frac{n_\theta}{M(\mathcal{X}(\eta_\theta))}e^{-\left(1+\I\right)\mathcal{X}(\eta_\theta) u_\theta^\mu p_\mu }\cr 
&\le \frac{1+C\sqrt{\mathcal{E}(f)(t)}}{M\left(\eta_0+C\sqrt{\mathcal{E}(f)(t)}\right)}e^{- \left(1+\I\right)\mathcal{X}\left( \eta_0-C\sqrt{\mathcal{E}(f)(t)}\right) \left(\sqrt{1+|u_\theta|^2}-|u_\theta| \right)  p^0}\cr
&\le \frac{1+C\sqrt{\mathcal{E}(f)(t)}}{M\left(\eta_0+C\sqrt{\mathcal{E}(f)(t)}\right)}e^{- \left(1+\I\right)\mathcal{X}\left( \eta_0-C\sqrt{\mathcal{E}(f)(t)}\right) \left(\sqrt{1+|C\mathcal{E}(f)(t)|^2}-C\sqrt{\mathcal{E}(f)(t)} \right) p^0}
\end{align*}
where we used the fact that (1) $u_\theta^\mu p_\mu$ is strictly positive:
$$
u_\theta^\mu p_\mu=\sqrt{1+|u_\theta|^2}\sqrt{1+|p|^2} -u_\theta \cdot p\ge 1
$$
due to the Cauchy-Schwarz inequality, and (2) a function $h(y):=\sqrt{1+y^2}-y$ is strictly decreasing on $y\in [0,\infty)$. We take $\mathcal{E}(f)(t)$ sufficiently small so that
$$
C_*:=\mathcal{X}\left( \eta_0-C\sqrt{\mathcal{E}(f)(t)}\right) \left(\sqrt{1+|C\mathcal{E}(f)(t)|^2}-C\sqrt{\mathcal{E}(f)(t)} \right)-\frac 12\mathcal{X}(\eta_0) >0.
$$
We then have 
$$
\frac{\F_\theta}{\sqrt{\F_0}}\le C e^{-C_*(1+\I)p^0},
$$
and similarly,
\begin{equation}\label{theta3}
\left|\frac{\pa^{\alpha}\F_\theta}{\sqrt{\F_0}}\right|\le C\left\{(1+\I)p^0\right\}^{|\alpha|} e^{-C_*(1+\I)p^0}.
\end{equation}
Finally, we combine \eqref{theta1} and \eqref{theta3} to obtain
\begin{align*}
&\left| \inti \intr \frac{1}{(1+\I)p^0\sqrt{\F_0}}\left|\pa^\alpha\nabla^2_{(n_\theta, u_\theta,\eta_\theta)}\mathcal{F}_\theta\right| g(p,\I) \phi(\I)\,d\I dp\right| \cr 
&\le   C\inti \intr \left\{1+(1+\I)p^0\right\}^{|\alpha|+1} e^{-C_*(1+\I)p^0} \left|g(p,\I)\right| \phi(\I)\,d\I dp \cr 
&\le C\|g\|_{L^2_{p,\I}}.
\end{align*}

\end{proof}

\begin{lemma}\label{nonlinear2}
Let $N\ge 3$.	Assume $f$ and $g$ are solutions to the perturbed relativistic BGK model \eqref{target_linearized}. For sufficiently small $\mathcal{E}(f)(t)$ and $\mathcal{E}(g)(t)$, we have
	$$
\left \langle  \Gamma(f)-\Gamma(g), f-g\right\rangle_{L^2_{p,\I}}\le
	C\left(\sqrt{\mathcal{E}(f)(t)}+\sqrt{E(g)(t)}\right)\|f-g\|^{2}_{L^2_{p,\I}}.$$
\end{lemma}
\begin{proof}
The proof is almost identical to the proof of Lemma \ref{nonlinear1} using the triangle inequality, so we omit it for brevity.	
\end{proof}
%%%%%%%%%%%%%%%%%%%%%%%%%%%%%%%%%%%%%%%%%%%%%%%%%%%%%%%%%%%%%%%%%%%%%%%%%%%%%%%%%%%%%%%%%%%%%%%%%%%%%%%%%%%%%%%%%%%%%%%%%%%%%%%%%%%%
%
%
%
%
%
%%%%%%%%%%%%%%%%%%%%%%%%%%%%%%%%%%%%%%%%%%%%%%%%%%%%%%%%%%%%%%%%%%%%%%%%%%%%%%%%%%%%%%%%%%%%%%%%%%%%%%%%%%%%%%%%%%%%%%%%%%%%%%%%%%%%%%

\section{Proof of theorem \ref{main}}
The main idea of the proof is the nonlinear energy method developed in \cite{Guo whole,Guo VMB} which is now standard. We first establish a local solution in a short time, and then extend it to the global one through the micro-macro decomposition using the Poincar\'{e} inequality. 
\subsection{Local existence}
Consider the iteration scheme:
\begin{align*}
\partial_{t}F^{n+1}+\frac{p}{p^0}\cdot\nabla_{x}F^{n+1}&=\frac{1}{(1+\I)p^{0}}\left(\F_E(F^n)-F^{n+1}\right),\crcr
F^{n+1}(0,x,p,\I)&=F_{0}(x,p,\I),
\end{align*}
with $F^0(t,x,p,\I)=F_0(x,p,\I)$. Then $F^{n+1}$ is given as
\begin{equation}\label{F n+1}
F^{n+1}(t,x,p,\I)=e^{-\frac{t}{(1+\I)p^0}}F_0\big(x-\frac{p}{p^0}t,p,\I\big)+\int_0^t \frac{1}{(1+\I)p^0}e^{\frac{s-t}{(1+\I)p^0}}\F_E(F^n)\big(s,x-\frac{p}{p^0}(t-s),p,\I\big)\,ds.
\end{equation}
On the other hand, letting $F^{n+1}=\F_0+\sqrt{\F_0}f^{n+1}$, we have shown in Lemma \ref{linearization} that $\{f^n\}$ verifies
\begin{align}\label{5}\begin{split}
\partial_{t}f^{n+1}+\frac{p}{p^0}\cdot\nabla_{x}f^{n+1}&=\frac{1}{(1+\I)p^0}\left\{\mathcal{P}_0(f^n)-f^{n+1} \right\}+\Gamma(f^n),\cr 
f^{n+1}(0,x,p,\I)&=f_{0}(x,p,\I).
\end{split}\end{align}
\begin{lemma}\label{lem9}
	There exist $M_{0}>0$ and $T_{*}>0$ such that if $\mathcal{E}(f_0)<\frac{M_{0}}{2}$ then $\mathcal{E}(f^{n})(t)<M_{0}$ implies $\mathcal{E}(f^{n+1})(t)<M_{0}$ for $t\in [0,T_{*}].$
\end{lemma}
\begin{proof}
For \eqref{5},	we take $\partial^{\alpha}$ and take weighted inner product $\langle\cdot,\cdot\rangle_{L^2_{x,p,\I}}$ with  $\partial^{\alpha}f^{n+1}$ to get
\begin{align*}
&\frac12 \frac{d}{dt}\|\pa^\alpha f^{n+1}\|^2_{L^2_{x,p,\I}}+\inti \intro \frac{1}{(1+\I)p^0}|\pa^\alpha f|^2\phi(\I)\,d\I dp dx\cr 
&=\big\langle \frac{1}{(1+\I)p^0}\pa^\alpha \mathcal{P}_0(f^n),\pa^\alpha f^{n+1}\big\rangle_{L^2_{x,p,\I}}+\left\langle  \pa^\alpha \Gamma(f^n),\pa^\alpha f^{n+1}\right\rangle_{L^2_{x,p,\I}},
\end{align*}
where we used \eqref{F n+1} and the induction hypothesis so that
\begin{align*}
\big\langle  \frac{p}{p^0}\cdot\nabla_{x}\pa^\alpha f^{n+1},\pa^\alpha f^{n+1}\big\rangle_{L^2_{x,p,\I}}&=\frac 12\inti \intro \nabla_x \cdot \frac{p}{p^0}\left\{ \pa^\alpha f^{n+1}\right\}^2\phi(\I)\,d\I dp dx=0.
\end{align*}
Recalling \eqref{rearrange}, we deduce
\begin{align*}
\left| \big\langle \frac{1}{(1+\I)p^0}\pa^\alpha \mathcal{P}_0(f^n),\pa^\alpha f^{n+1}\big\rangle_{L^2_{x,p,\I}}\right|&\le C \into \|\pa^\alpha f^{n}\|_{L^2_{p,\I}} \left\{ \inti \intr \sqrt{\F_0}|\pa^\alpha f^{n+1}| \phi(\I)\, d\I dp\right\}dx\cr
&\le C \into \|\pa^\alpha f^{n}\|_{L^2_{p,\I}} \|\pa^\alpha f^{n+1}\|_{L^2_{p,\I}} dx\cr 
&\le C\left( \mathcal{E}(f^n)(t)+\mathcal{E}(f^{n+1})(t)\right),
\end{align*}
which together with Lemma \ref{nonlinear1} gives
\begin{align}\label{energy local}
& \frac{d}{dt}\|\pa^\alpha f^{n+1}\|^2_{L^2_{x,p,\I}}\le  C\left( \mathcal{E}(f^n)(t)+\mathcal{E}(f^{n+1})(t)\right)
\end{align}
where we have assumed $M_0$  is sufficiently small. Summing over $\alpha$ and integrating in time, we get
	\begin{align*}
	\mathcal{E}(f^{n+1})(t)&\le \mathcal{E}(f^{n+1})(0)+C\int_{0}^{t}\mathcal{E}(f^{n})(s)+\mathcal{E}(f^{n+1})(s)\,ds\cr
	&\le\frac{M_{0}}{2}+CT_{\ast}\left(\sup_{0\le s\le T_{\ast}}\mathcal{E}(f^{n})(s)+\sup_{0\le s\le T_{\ast}}\mathcal{E}(f^{n+1})(s)\right)
	\end{align*}
	which yields
	$$
	\left(1-CT_{\ast}\right)\sup_{0\le s\le T_{\ast}}\mathcal{E}(f^{n+1})(s)\le\left(\frac{1}{2}+CT_{\ast}\right)M_{0}.
	$$
	This gives the desired result for sufficiently small $T_{\ast}$.
\end{proof}
 \begin{theorem}\label{thm8}
	Let $N\geq3$. Assume that   $F_{0}=\mathcal{F}_0+\sqrt{\mathcal{F}_0}f_{0}$ is non-negative and it shares the same total mass, momentum, and energy as the global equilibrium $\mathcal{F}_0$, i.e. 
	\begin{equation}\label{initial condition}
	\inti\intro f_{0}\sqrt{\mathcal{F}_0} \phi(\I)\,d\I dpdx=0,\quad \inti\intro  (1+\I)p^\mu f_{0}\sqrt{\mathcal{F}_0} \phi(\I)\,d\I dpdx=0.
	\end{equation}
	Then there exist $M_{0}>0$ and $T_{*}>0$, such that if $T_{*}\le\frac{M_{0}}{2}$ and $\mathcal{E}(f_0)\le \frac{M_{0}}{2}$, there is a unique solution $F(t,x,p,\I)$ to the relativistic BGK \eqref{PMarle} such that
	\begin{enumerate}
		\item[(1)] The energy functional is continuous on $[0,T_{*}]$ and uniformly bounded:
		$$
		\sup_{0\le t\le T_{*}}\mathcal{E}(f)(t)\le M_{0}.
		$$
		\item[(2)] The distribution function remains positive in $[0,T_{*}]$:
		$$
		F(t,x,p,\I)=\mathcal{F}_0+\sqrt{\mathcal{F}_0}f(t,x,p,\I)\ge 0.
		$$
		\item[(3)] The total mass, momentum, and energy of perturbation parts vanishes:
		 $$
		 \inti\intro f\sqrt{\mathcal{F}_0} \phi(\I)\,d\I dpdx=0,\quad \inti\intro  (1+\I)p^\mu f\sqrt{\mathcal{F}_0} \phi(\I)\,d\I dpdx=0.
		 $$
		%\item[(3)] The conservation laws \eqref{4} hold for all $[0,T_{*}].$
	\end{enumerate}
\end{theorem}
\begin{proof}
Due to the Rellich–Kondrachov compactness theorem for periodic domains, Lemma \ref{lem9} implies the existence of local-in-time  solutions to the iteration scheme \eqref{5} satisfying (1). The continuity of $\mathcal{E}$ can be easily obtained by \eqref{energy local} due to the uniform bound (1). The non-negativity of the distribution function can be guaranteed by the relation \eqref{F n+1} and the initial assumption $F_0\ge 0$.  To prove the uniqueness of the solution, we assume that  $\tilde{F}=\mathcal{F}_0+\sqrt{\mathcal{F}_0}\tilde{f}$ is another local solution corresponding to the same initial data $F_0$. Letting $g=f-\tilde{f}$, then $g$ verifies
\begin{align*}
\partial_{t}g+\frac{p}{p^0}\cdot\nabla_{x}g&=\frac{1}{(1+\I)p^0}\left\{\mathcal{P}_0(g)-g \right\}+\Gamma(f)-\Gamma(\tilde{f}),\cr 
g(0,x,p,\I)&=0.
\end{align*}
	Taking the weighted inner product $\langle \cdot,\cdot\rangle_{L^2_{x,p,\I}}$ with $g$ and applying Proposition \ref{linear property} and Lemma \ref{nonlinear2}, we get 
	\begin{align*}
	\frac{1}{2}\frac{d}{dt}\|g\|^{2}_{L^2_{x,p,\I}}+C\|(I-\mathcal{P})g\|_{L^2_{x,p,\I}}\le C\left\{\big(\mathcal{E}(f)(t)\big)^{\frac12}+\big(\mathcal{E}(\tilde{f})(t)\big)^{\frac12}\right\}\|g\|^{2}_{L^2_{x,p,\I}},	\end{align*}
and thus
	$$
\|g(t)\|^{2}_{L^2_{x,p,\I}}\le  e^{Ct}\|g(0)\|^{2}_{L^2_{x,p,\I}}
	$$
	which gives the uniqueness of solutions. Finally, since the local solution is smooth, it follows from the relation \eqref{cancellation} that
	$$
	\frac{d}{dt}\inti\intro f\sqrt{\mathcal{F}_0} \phi(\I)\,d\I dpdx=0,\quad \frac{d}{dt}\inti\intro  (1+\I)p^\mu f\sqrt{\mathcal{F}_0} \phi(\I)\,d\I dpdx=0
	$$
	which together with \eqref{initial condition} gives (3).
\end{proof}

\subsection{Global existence and asymptotic behavior of solutions}

Recall from \eqref{Pf} that $\mathcal{P}$ is an orthonormal projection defined by
$$
\mathcal{P}(f)=\langle f,e_1\rangle_{L^2_{p,\I}} e_1+\langle f,e_{2,3,4} \rangle_{L^2_{p,\I}} \cdot e_{2,3,4}+ \langle f,e_{5}\rangle_{L^2_{p,\I}} e_5,
$$
where
$$
e_1=\sqrt{\F_0},\qquad  e_{1+i}=\frac{(1+\I)p^i\sqrt{\F_0}}{\|(1+\I)p^i\sqrt{\F_0}\|_{L^2_{p,\I}}},  \qquad e_5=\frac{(1+\I)\left(p^0-\delta\right)\sqrt{\F_0}}{\|(1+\I)(p^0-\delta )\sqrt{\F_0}\|_{L^2_{p,\I}}}.
$$
This can be expressed in the form of 
$$
\mathcal{P}(f)=a_f(t,x)\sqrt{\F_0}+b_f(t,x)\cdot (1+\I)p\sqrt{\F_0}+c_f(t,x) \left\{(1+\I)p^0-\delta\right\}\sqrt{\F_0} 
$$
Note from Proposition \ref{linear property} that  $Ker(\mathcal{L})=\{\sqrt{\mathcal{F}_0},(1+\I)p^\mu\sqrt{\mathcal{F}_0} \}$, yielding $\mathcal{L}[\mathcal{P}(f)]=0$, so substituting $f=\mathcal{P}(f)+\{I-\mathcal{P}\}(f)$ into \eqref{target_linearized} leads to
\begin{align*}
\left\{\partial_{t}+\frac{p}{p^0}\cdot\nabla_{x}\right\}\mathcal{P}(f)&=\left\{-\partial_{t}-\frac{p}{p^0}\cdot\nabla_{x}+\mathcal{L}\right\}\{I-\mathcal{P}\}(f)+\Gamma (f)\cr
&=: l\{I-\mathcal{P}\}(f)+h(f).
\end{align*}
Computing the left-hand side, we get
\begin{align*}
&\left\{\partial_{t}+\frac{p}{p^0}\cdot\nabla_{x}\right\}\mathcal{P}(f)\cr
&=\left\{\partial_{t}+\frac{p}{p^0}\cdot\nabla_{x}\right\}\left\{a_f\sqrt{\F_0}+b_f\cdot (1+\I)p\sqrt{\F_0}+c_f \left\{(1+\I)p^0-\delta\right\}\sqrt{\F_0}\right\}\cr
&=\big(\partial_{t}\tilde{a}_f\big)e_{a_0}+\sum_{i=1}^{3}\big(\partial_{x_{i}}\tilde{a}_f\big)e_{a_i}+\sum_{i=1}^{3} (\partial_{t}b_{f_i}+\partial_{x_{i}}c_f)e_{bc_i}
+\sum_{j=1}^{3}\sum_{i=1}^{3}(\partial_{x_{i}}b_{f_j})e_{ij}+(\partial_{t}c_f)e_c
\end{align*}
where 
$$
\tilde{a}_f:=a_f-\delta c_f
$$ 
and  $e_{a_0},\cdots, e_{c}$ denote 
\begin{equation}\label{basis}
\{e_{a_0}, e_{a_i}, e_{bc_i}, e_{ij}, e_{c}\}
=\Big\{\sqrt{\F_0},\ \frac{p^{i}}{p^{0}}\sqrt{\F_0},\ (1+\I)p^{i}\sqrt{\F_0},\ (1+\I)\frac{p^{i}p^{j}}{p^{0}}\sqrt{\F_0},\ (1+\I)p^{0}\sqrt{\F_0}\Big\}.
\end{equation}
  Let $l_{a_0}\cdots l_{c}$ and $h_{a_0},\cdot,h_{c}$ denote the coefficients of $l\{(I-\widetilde{P})(f))\}$
	and $h(f)$ with respect to (\ref{basis}) respectively. Then we have
	\begin{enumerate}
		\item $\partial_{t}\tilde{a}=l_{a_0}+h_{a_0}$
		\item $\partial_{t}c=l_{c}+h_{c}$
		\item $\partial_{t}b_{i}+\partial_{x_{i}}c=l_{bc_i}+h_{bc_i}$
		\item $\partial_{x_{i}}\tilde{a}=l_{ai}+h_{ai}$
		\item $(1-\delta_{ij})\partial_{x_{i}}b_{j}+\partial_{x_{j}}b_{i}=l_{ij}+h_{ij}$
	\end{enumerate}
Notice that the system above is exactly the same as in the case of the relativistic BGK model of Anderson and Witting, see \cite{HY2}. Thus, following the standard analysis developed in \cite{Guo whole,Guo VMB} gives
$$
\sum_{|\alpha|\le N}\left\|\pa^\alpha \mathcal{P}(f) \right\|^2_{L^2_{x,p,\I}} \le C \sum_{|\alpha|\le N}\left\{\left\|\pa^\alpha \{I-\mathcal{P}\}(f) \right\|^2_{L^2_{x,p,\I}}+\sqrt{\mathcal{E}(f)(t)} \left\|\pa^\alpha f \right\|^2_{L^2_{x,p,\I}}\right\}
$$
where we used the Poincar\'{e} inequality to the lowest order term in the micro-macro decomposition thanks to Theorem \ref{thm8} (3). This, combined with the coercive property presented in Proposition \ref{linear property} (3), yields
	\begin{equation}\label{coercive}
	\sum_{|\alpha|\le N}\langle \mathcal{L}\partial^{\alpha}f,\partial^{\alpha}f\rangle_{L^2_{x,p,\I}}\le-\lambda_0\sum_{|\alpha|\le N}\|\partial^{\alpha}f\|^{2}_{L^2_{x,p,\I}}
	\end{equation}
	for some positive constant $\lambda_0$. Now we apply $\partial^{\alpha}$ to (\ref{target_linearized}), take an inner product with $\partial^{\alpha}f$ with respect to $\langle\cdot,\cdot\rangle_{L^2_{x,p,\I}}$, and use Lemma \ref{nonlinear1} and \eqref{coercive} to obtain 
\begin{align*}
&\frac{1}{2}\frac{d}{dt}\sum_{|\alpha|\le N}\|\partial^{\alpha} f\|^{2}_{L^2_{x,p,\I}}+\lambda_0 \sum_{|\alpha|\le N}\|\partial^{\alpha} f\|^{2}_{L^2_{x,p,\I}}\le  C\sqrt{\mathcal{E}(f)(t)}\mathcal{E}(f)(t),
\end{align*}
i.e.
$$
\frac{d}{dt}\mathcal{E}(f)(t)+2\lambda_0\mathcal{E}(f)(t) \le  C\sqrt{\mathcal{E}(f)(t)}\mathcal{E}(f)(t).
$$
	We define
	$$
	M=\min\biggl\{\frac{\lambda_0^2}{C^2},M_0\biggl\}
	$$
	and choose the initial data sufficiently small in the following sense:
	$$
	\mathcal{E}(f_{0})\le  \frac{M}{2} < M_{0}.
	$$
	Let $T>0$ be given as
	$$
	T=\sup_{t}\big\{t:\mathcal{E}(f)(t)\le M\big\}
	$$
	which implies
	$$
	\mathcal{E}(f)(t)\le M \le  M_{0}.
	$$
	We then have for $0\le t\le T$
	\begin{align}\label{timedecay}
	\begin{split}
\frac{d}{dt}\mathcal{E}(f)(t)+2\lambda_0\mathcal{E}(f)(t) \le  C\sqrt{\mathcal{E}(f)(t)}\mathcal{E}(f)(t)\le \lambda_0 \mathcal{E}(f)(t),
	\end{split}\end{align}
and hence
	\begin{align*}
\mathcal{E}(f)(t)+\lambda_0 \int_0^t\mathcal{E}(f)(s)\,ds \le \mathcal{E}(f_0) \le \frac{M}{2}<M_0,
	\end{align*}
Due to the continuity of $\mathcal{E}(f)(t)$,  we conclude that $T=\infty$. Finally, applying Gr\"{o}nwall's inequality to  \eqref{timedecay} gives the exponential decay:
$$
\mathcal{E}(f)(t)\le  e^{-\lambda_0 t}\mathcal{E}(f_0)
$$ 
which completes the proof.

	\section*{Acknowledgments}
		This research was funded by a 2023 Research Grant from Sangmyung University (2023-A000-0285).
	
	%%%%%%%%%%%%%%%%%%%%%%%%%%%%%%%%

	%	\section*{Acknowledgements}

	%%%%%%%%%%%%%%%%%%%%%%%%%%%%%%%%%%%%%%%%%%%%%%%%%%%%%%%%%%%%%%%%%%%%%%%%%%%%%%%%%
	%
	%
	%                        thebibliography
	%
	%
	%%%%%%%%%%%%%%%%%%%%%%%%%%%%%%%%%%%%%%%%%%%%%%%%%%%%%%%%%%%%%%%%%%%%%%%%%%%%%%%%%

	\bibliographystyle{amsplain}

\end{document}